\documentclass[11pt]{amsart}
\usepackage{fullpage}

\usepackage{hyperref}

\usepackage{amsmath,amsfonts,amssymb}
\usepackage{mathrsfs}
\usepackage{verbatim}
\usepackage{amsthm}
\usepackage{mathrsfs} 
\usepackage{color}

\newtheorem{theorem}{Theorem}[section]
 \newtheorem{corollary}[theorem]{Corollary}
 \newtheorem{lemma}[theorem]{Lemma}
 \newtheorem{proposition}[theorem]{Proposition}
 \theoremstyle{definition}
 \newtheorem{definition}[theorem]{Definition}
 \newtheorem{setting}[theorem]{Setting}
 \theoremstyle{remark}
 \newtheorem{remark}[theorem]{Remark}
  
 \numberwithin{equation}{section}

\def \bC {\mathbb C}

\def \bN {\mathbb N}

\def \bR {\mathbb R}
\def \bS {\mathbb S}
\def \bT {\mathbb T}

\def \bZ {\mathbb Z}

\def \cD {\mathcal D}
\def \cE {\mathcal E}
\def \cF {\mathcal F}

\def \cH {\mathcal H}

\def \cL {\mathcal L}
\def \cM {\mathcal M}

\def \cO {\mathcal O}

\def \cS {\mathcal S}

\def \sL{\mathscr L}
\def \tr {\text{\rm tr}}
\def \TR {\text{\rm TR}}
\def \id {\text{\rm I}}
\def\Op{{{\rm Op}}}
\def\supp{{{\rm supp}}}
\def\res{{\text{\rm res}}}
\def \sig {\varsigma}
\def \Dom {\text{\rm Dom}}

\begin{document}

\title
{Real trace expansions}
\author
{V\'eronique Fischer}
\address
{Department of Mathematical Sciences\\
University of Bath\\
Claverton Down\\ 
Bath  BA2 7AY, United Kingdom}
\email{v.c.m.fischer@bath.ac.uk}

\date{October 2019}

\keywords{Pseudodifferential operators on manifolds, non-commutative residues, canonical trace}

\subjclass[2010]{58J40, 58J42}

\begin{abstract}
In this paper, we investigate trace expansions of operators of the form $A\eta(t\cL)$ 
where $\eta:\bR\to \bC$ is a Schwartz function, 
$A$ and $\cL$ are classical pseudo-differential operators on a compact manifold $M$ with  $\cL$  elliptic.
In particular, we show that, under certain hypotheses, this trace admits an expansion in powers of $t\to 0^+$.
We also relate the constant coefficient to the non-commutative residue and the canonical trace of $A$. 
Our  main tool is the continuous inclusion of the functional calculus of $\cL$ into the pseudo-differential calculus whose proof  relies on the Helffer-Sj\"ostrand formula.
\end{abstract}

\maketitle

\makeatletter
\renewcommand\l@subsection{\@tocline{2}{0pt}{3pc}{5pc}{}}
\makeatother

\tableofcontents

\section{Introduction}

Trace expansions of operators are fundamental objects in geometric analysis, 
especially in index theory, spectral geometry and related mathematical physics, see e.g.  \cite[Ch. 5]{gilkey}
  and \cite{lesch_2010}.
In this paper, we investigate trace expansions of operators of the form $A\eta(t\cL)$ 
where $\eta:\bR\to \bC$ is a Schwartz function, 
$A$ and $\cL$ are classical pseudo-differential operators on a compact manifold $M$ and  $\cL$ is elliptic.
We will show that the trace admits an expansion in all the powers of $t\to 0^+$ when $\eta$ is supported away from 0, but only in the positive powers when $\eta$ is supported near 0, see Theorem \ref{thm_intro} below.

Another result of this paper is the identification of the constant coefficient in the expansion:
it is  equal to the non-commutative residue $\res(A)$ (up to a known constant of $\eta$) 
when $\eta$ is supported away from 0,
and to the canonical trace $\TR(A)$
when $\eta(0)=1$.
The definitions of the non-commutative residue and of the canonical trace will be recalled in Sections \ref{subsec_def_res} and \ref{subsubsec_def_TR} together with further references, so in this introduction  
we restrict our comments to their origins and uses.
The non-commutative residue was introduced independently by Guillemin \cite{guillemin} 
and Wodzicki \cite{Wod_84,Wod_87} in the early eighties.
Beside being the only trace on the algebra of pseudo-differential operators on $M$ up to constants, 
its importance comes from its applications in mathematical physics, 
mainly in Connes' non-commutative geometry due to its link with the Dixmier trace \cite{connes}
but also in relation with e.g. the Einstein-Hilbert action
(see \cite[Section 6.1]{lesch_2010} and the references therein).
The canonical trace was constructed by Kontsevich and Vishik in the mid-nineties \cite{KV} as a tool to study further zeta functions and determinants of elliptic pseudo-differential operators. Since then, it 
has received considerable attention and found interesting applications, 
see e.g. \cite{scott,paycha,okikiolu}.

Beside the expansion in itself and the identification of its constant coefficient, 
this paper proposes a new strategy to tackle these questions.
Indeed, the first trace expansions may be found in works by  
H\"ormander \cite{hormander2} and Duistermaat and Guillemin \cite{duistermaat+guillemin}, 
using Fourier integral operators, 
and  also by Seeley \cite{seeley_67, seeley_69, seeley_69_2},
using pseudo-differential calculi depending on a complex parameter (for the latter, see also \cite{grubb,Grubb+Schrohe_2001,Grubb+Seeley,lesch,schrohe,shubin_bk}).
The most studied examples of functions $\eta(\lambda)$ are
$e^{-z\lambda}$ and $\lambda^z$ where $z$ is a complex parameter
leading to 
the heat trace and the zeta function respectively.
The proofs using Seeley's ideas, in particular the relations between all these well-known expansions,  rely on the meromorphy in the complex parameter, for instance on contour integration which allows this complex parameter to lie far away from the spectrum of the elliptic operator.

Proving expansions of $\tr (A \eta(t\cL))$  using the ideas and results of \cite{hormander2} and \cite{duistermaat+guillemin} requires  e.g. that the Fourier transform of $\eta$ is compactly supported near 0, 
see \cite{Li+Stroh}.
Our proof will not use  pseudo-differential calculi with a complex parameter away from the spectrum of $\cL$ or Fourier integral operators or previous results in these directions.
Indeed,
our main tool is the continuous inclusion of the functional calculus of formally self-adjoint elliptic operators $\cL$ into the pseudo-differential calculus. 
We show this inclusion  in Section \ref{sec_FC}
using  the Helffer-Sj\"ostrand formula.
Our method requires the elliptic operator $\cL$ to be formally self-adjoint 
whereas Seeley's allows for a complex elliptic operator
whose principal symbol is non-negative (or more generally does not take values in a half-line of the complex plane).
 
 \medskip
 
The main result of this paper is summarised in the following statement:
\begin{theorem}
\label{thm_intro}
Let $\cL \in \Psi_{cl}^{m_0}(M)$ be a pseudo-differential operator on  a compact manifold $M$ which is elliptic and formally self-adjoint in the sense of Setting \ref{set_cL_M}.
Let $A\in \Psi^m_{cl}(M)$ and let $\eta:\bR\to \bC$ be a Schwartz function.
Then the operator $A\ \eta(t \cL)$ is traceclass for all $t\in \bR$.
\begin{enumerate}
\item 
If $\eta$ is supported away from 0, then the trace  of $A\ \eta(t \cL)$ admits the following expansion as $t\to 0^+$, 
$$
\tr \left(A \eta(t\cL)\right)
\sim 
c_{m+n}
t^{-\frac{m+n}{m_0}} 
+
c_{m-n-1}
t^{-\frac{m+n-1}{m_0}} 
+
\ldots 
$$	
If in addition  $m \in \bZ_n=\{-n,-n+1,-n+2,\ldots\}$, then the $t^0$-coefficient $c_0$ 
is  the non-commutative residue up to a known constant of $\cL$ and $\eta$.
\item 	
If  $m_0 \in \bN$
and if $m\in \bC \setminus\bZ_n$ with $\Re m>-n$,
then we have as $t\to 0$
 $$
 \tr (A\eta (t\cL)) = \eta(0)\, \TR(A)+ \sum_{j=0}^{N-1} c_{m+n-j} t^{\frac{-m-n+j}{m_0}}  +o(1),
 $$
where  $N\in \bN$ is  such that $\Re m+n<N$.
\end{enumerate}
\end{theorem}

A  version of Part (1) with additional information on the constants $c_{m+n-j}$ is given in Theorem~\ref{thm_main}.
They are the same constants appearing in Part (2) for which a more detailed statement is given in Theorem \ref{thm_TR}.
These constants are local in the sense that they depend only on the poly-homogeneous expansions of the symbol of $A$ in local charts. 
The canonical trace $\TR$ is not a local constant, and we think that Part(2) above  sheds light on the interpretation of $\TR$ as a 
 generalisation of  the classical Hadamard partie finie regularisation of integrals (see Section \ref{subsubsec_statement_TR}).
 In  \cite{fischer_locglob}, 
  the results of this paper are given in a more explicit way when $M$ is a compact Lie group and $\cL$ the Laplace-Beltrami operator:  in this setting, a notion of full symbols of pseudo-differential operators can be defined globally, and 
 the interpretation of $\TR$ as `Hadamard partie finie' is even more apparent.

In this paper, we consider the self-adjoint extension of the  elliptic operator $\cL$ on a compact manifolds $M$ without boundary or with Dirichlet boundary condition on an open set $\Omega$ of $\bR^n$. 
It will be interesting to extend our result and methods to the case of compact manifolds $M$ with boundary
 as in e.g. \cite{Grubb+Schrohe_2001,schrohe}.
We also leave the applications  of these expansions in geometry and index theory  to future works.
For the sake of clarity and brevity of the paper, 
we consider 
only pseudo-differential symbols with scalar values,
although  the generalisation to pseudo-differential operators acting on sections of a vector bundle over $M$ is straightforward for our results and proofs.

\smallskip

The paper is divided as follows.
The next section is devoted to present our main  results in their full generality;
this includes defining the settings of our investigations,
and recalling 
the definitions of our main objects, especially the non-commutative residue and the canonical trace.
The proofs of our result rely on the continuous inclusions of the functional calculus of elliptic operators 
 into the pseudo-differential calculus (presented in Section \ref{sec_FC})
 and  are given in Section \ref{sec_pf_thms}.

\smallskip

\textbf{Acknowledgements.}
The author thanks Alexander Strohmaier for bringing important literature to her attention and Jean-Marc Bouclet for his excellent online lecture notes \cite{bouclet}, as well as the anonymous referees for invaluable comments.

\section{Settings and statements}
\label{sec_res}

In this section, we discuss trace expansions in relations with the non-commutative residue.
After setting the notation for the pseudo-differential calculus and the H\"ormander classes 
(Section \ref{subsec_PsiOmega}), 
we recall the definition of the non-commutative residue via local symbols (Section \ref{subsec_def_res}), 
present its relations with well-known tracial expansions
(Section \ref{subsec_heatexp})
and with our new expansion 
(Sections \ref{subsec_statement_res} and \ref{subsec_cor}).
We end this section with the case of the canonical trace (Section \ref{subsec_TR}).

\subsection{The H\"ormander classes}
\label{subsec_PsiOmega}

Here we recall some well-known facts and set some notations for the H\"ormander pseudo-differential calculus.
Classical references for this material include \cite{Alinhac+Gerard,lerner_notes,shubin_bk,taylor_bk_princeton}.
We will not recall or make use of the links of non-commutative residue with zeta functions 
\cite{Wod_84,Wod_87,kassel} or with the Dixmier trace
\cite{connes}.

\subsubsection{The pseudo-differential calculus on $\bR^n$}
\label{subsec_defPDO}

We denote by $S^m(\bR^n\times \bR^n)$ the H\"ormander class of symbols of order $m\in \bR$ on $\bR^n$, 
that is, the Fr\'echet space of smooth functions $a:\bR^n\times \bR^n\to \bC$ satisfying 
$$
\|a\|_{S^m,N}:=
\sup_{\substack{
\alpha,\beta\in\bN_0^n\\ |\alpha|,|\beta|\leq N}}
\sup_{(x,\xi) \in \bR^n\times\bR^n}
\langle \xi \rangle^{|\alpha|-m}|\partial_x^\beta \partial_\xi^\alpha a(x,\xi)|<\infty;
$$ 
in this paper, 
$\bN_0=\{0,1,2,\ldots\}$ denotes the set of non-negative integers
and 
$\bN=\{1,2,\ldots\}$ the set of positive integers;
we also   use the usual notation 
$\langle \xi \rangle = \sqrt{1+|\xi|^2}$,
and
$\partial_j=\partial_{x_j}$ for the partial derivatives in $\bR^n$
as well as
$\partial^\alpha=\partial_1^{\alpha_1}\partial_2^{\alpha_2}\ldots$ etc.

To each symbol $a\in S^m(\bR^n\times \bR^n)$, 
we associate the operator $\Op(a)$ defined via
\begin{equation}
\label{eq_Opaf}	
\Op(a) f(x) = \int_{\bR^n} \widehat f(\xi) e^{2i\pi x\cdot \xi} a(x,\xi) d\xi, \qquad x\in \bR^n, \ f\in \cS(\bR^n).
\end{equation}
Here, $\widehat f$ denotes the Euclidean Fourier transform of $f\in \cS(\bR^n)$:
$$
 \widehat f(\xi) = \cF_{\bR^n} f(\xi) =\int_{\bR^n} f(x) e^{-2i\pi x\cdot \xi} dx.
 $$
We denote by $\Psi^m(\bR^n)= \Op(S^m(\bR^n\times \bR^n))$   
the H\"ormander class of operators of order $m\in \bR$ on $\bR^n$.
Recall that $\Op$ is one-to-one on $S^m$ 
and thus that $\Psi^m(\bR^n)$ inherit a structure of Fr\'echet space.

The class of smoothing symbols  $S^{-\infty}=\cap_{m\in \bR} S^m$ is equipped with the projective limit topology. The class of smoothing symbols is denoted by $\Psi^{-\infty} (\bR^n)=\Op(S^{-\infty})$.

The space $\cup_{m\in \bR} \Psi^m(\bR^n)$ is an algebra of operators, and furthermore the composition mapping $(T_1,T_2)\mapsto T_1T_2$ is continuous $\Psi^{m_1}(\bR^n)\times \Psi^{m_2}(\bR^n)\to \Psi^{m_1+m_2}(\bR^n)$.
This algebra of operators is also 
stable under taking the adjoint with 
taking the formal adjoint $T\mapsto T^*$  being continuous $\Psi^m(\bR^n)\to \Psi^m(\bR^n)$.
It acts on e.g. Sobolev spaces since an operator $T\in \Psi^m(\bR^n)$ maps the Sobolev spaces $H^s$ to $H^{s-m}$ boundedly.

\subsubsection{Expansions of symbols in $\bR^n$}
\label{subsec_exp}

A symbol $a\in S^m(\bR^n\times \bR^n)$ admits an \emph{expansion} when $\alpha_{m-j}\in S^{m-j}(\bR^n\times \bR^n)$ and 
$a-\sum_{j=0}^{N} \alpha_{m-j} \in S^{m-N-1}(\bR^n\times \bR^n)$. We then write $a\sim \sum_{j\in \bN_0} \alpha_{m-j}$
Any such expansion yields a symbol and such an expansion characterises a symbol modulo $S^{-\infty}$:
\begin{lemma}
\label{lem_expansion}
Let $m\in \bR$. Given a sequence of functions
 $(\alpha_{m-j})_{j\in \bN_0}$ satisfying
$\alpha_{m-j}\in S^{m-j}(\bR^n\times \bR^n)$, there exists a symbol $a\in S^m(\bR^n\times \bR^n)$ admitting the expansion $\sum_{j\in \bN_0} \alpha_{m-j}$.
Furthermore, if another symbol $b\in S^m(\bR^n\times \bR^n)$ admits  the same expansion, 
then $a-b\in S^{-\infty}(\bR^n\times \bR^n)$.
\end{lemma}
\begin{proof}
	We fix  
$\psi\in C^\infty(\bR^n)$ with $\psi(\xi)=0$ for $|\xi|\leq 1/2$ and $\psi(\xi)=1$ for $|\xi|\geq 1$. For $t>0$ we set $\psi_t(\xi)=\psi(t\xi)$ and we observe
$$
\|\alpha_{m-j} \psi_t\|_{S^{m},N}
\leq
C t^{j}
\|\alpha_{m-j}\|_{S^{m-j},N}\quad \mbox{with}\ C=C_{N,\psi}>0 .
$$
Therefore, choosing a sequence $(t_j)_{j\in \bN_0}$ converging to 0 and satisfying  $C_{N,\psi}  t_j^{j} \|\alpha_{m-j}\|_{S^{m-j},j}\leq 2^{-j}$ with $N=j$ for each $j\in\bN_0$, the sum $a=\sum_{j=0}^{+\infty} \alpha_{m-j} \psi_{t_j}$ is converging in the Fr\'echet space $S^m$ with 
$$
\|a - \sum_{j=0}^M \alpha_{m-j} \psi_{t_j} \|_{S^{m},N}
\leq
\sum_{j=M+1}^{+\infty} \|\alpha_{m-j} \psi_{t_j} \|_{S^{m},N}
\leq 
2^{-M},
$$
for $M\geq N$.
The rest of the statement follows easily.
\end{proof}

Expansions are useful when considering operations such as compositions and taking the adjoint modulo smoothing operators.
Indeed, if $a\in S^{m_a}(\bR^n\times \bR^n)$ and $b\in S^{m_b}(\bR^n\times \bR^n)$ then 
	$\Op(a)\Op(b)\in \Psi^{m_a+m_b}(\bR^n)$
	with symbol $a\# b$ having the expansion
$$
	a\# b \sim  
	\sum_{j\in \bN_0}\sum_{|\alpha|=j} \frac {(2i\pi)^{-|\alpha|}}{\alpha !} \partial_\xi^\alpha a \ \partial_x^\alpha b. 
$$

A symbol $a\in S^m(\bR^n\times \bR^n)$ with $m\in \bR$ admits a \emph{poly-homogeneous expansion with complex order} $\tilde m$ 
when it admits an expansion $a\sim \sum_{j\in \bN_0} \alpha_{m-j}$ 
where each function  $\alpha_{m-j}(x,\xi)$ is  $(\tilde m-j)$-homogeneous in $\xi$ for $|\xi|\geq 1$; here $\tilde m\in \bC$ with $\Re \tilde m=m$ and the homogeneity for $|\xi|\geq 1$ means that $\alpha_{m-j}(x,\xi) = |\xi|^{\tilde m-j}\alpha_{m-j}(x,\xi/|\xi|)$ for any $(x,\xi)\in \bR^n\times\bR^n$ with $|\xi|\geq 1$. 
We write the poly-homogeneous expansion as $a\sim_h \sum_j a_{\tilde m-j} $ where $a_{\tilde m-j}(x,\xi)=|\xi|^{\tilde m-j}\alpha_{m-j}(x,\xi/|\xi|)\in C^\infty (\bR^n \times (\bR^n\backslash\{0\}))$ is homogeneous of degree $\tilde m-j$ in $\xi$.  
We may call $a_{\tilde m}$ the (homogeneous) principal symbol of $a$ or of $A$ and $\tilde m$ the complex order of $a$ or $A$.

\subsubsection{The pseudo-differential calculus on an open set $\Omega\subset \bR^n$}

In this paper, $\Omega$ always denotes a non-empty open subset of $\bR^n$.

We denote by $S^m_{loc}(\Omega\times \bR^n)$
the set of smooth functions $a\in C^\infty(\Omega\times \bR^n)$ such that $\chi a\in S^m(\bR^n\times \bR^n)$
for any cut-off $\chi\in C_c^\infty(\Omega)$.
Equivalently, this is 
the Fr\'echet space of smooth functions $a:\bR^n\times \bR^n\to \bC$ satisfying 
$$
\|a\|_{S^m_{loc},N}:=
\sup_{\substack{
\alpha,\beta\in\bN_0^n\\ |\alpha|,|\beta|\leq N}}
\sup_{(x,\xi) \in K_N\times\bR^n}
\langle \xi \rangle^{|\alpha|-m}|\partial_x^\beta \partial_\xi^\alpha a(x,\xi)|<\infty,
$$ 
where $(K_N)_{N\in \bN_0}$ is an increasing sequence of non-empty compact subsets of $\Omega$ satisfying $\cup_{N\in \bN_0}K_N=\Omega$.

We say that $a\in S^m_{loc}(\Omega\times \bR^n)$ is $x$-compactly supported in $\Omega$ 
when there exists a compact $K$ containing the supports of all the functions $x\mapsto a (x,\xi)$, $\xi\in \bR^n$.

To each symbol $a\in S^m_{loc}(\Omega\times \bR^n)$, 
we associate the operator $\Op(a)$ defined via
\eqref{eq_Opaf}	for $x\in \Omega$.
This defines a continuous linear operator, still denoted by $\Op(a)$, 
from $C_c^\infty(\Omega)$ to $C^\infty(\Omega)$, 
from the space $\cE'(\Omega)$ of compactly supported distributions of $\Omega$ to the space $\cD'(\Omega)$ of distributions of $\Omega$, 
from the space $H^{s}_{comp}(\Omega)$ of  compactly supported distribution in $H^s$ to the space $H^{s-m}_{loc}(\Omega)$ of distributions which are locally in $H^{s-m}$;
recall that the Fr\'echet space $H^s_{loc}(\Omega)$
is equipped with the semi-norms given by 
$$
\|f\|_{H^s_{loc},N} = \|\chi_N f\|_{H^s_{loc}}
$$
for some sequence of functions $(\chi_N)\subset C_c^\infty(\Omega)$ with $\chi_N =1$ on the compact set $K_N$ as above.

Moreover, the map $a\mapsto \Op(a)$ is continuous from to $S^m_{loc}(\Omega\times \bR^n)$ to the topological space of continuous linear maps from $H^s_{comp}(\Omega)$ to $H^{s-m}_{loc}(\Omega)$.
We denote by $\Psi^m(\Omega):=\Op(S^m_{loc}(\Omega))$ the resulting Fr\'echet space of operators.

By the Schwartz kernel theorem, $\Op(a)$ has an integral kernel $K \in \cD'(\Omega\times \Omega)$ given formally via
$$
\Op(a)f(x)
=\int_{\Omega} K(x,y) f(y) dy, \qquad f\in C_c^\infty(\Omega).
$$
Recall that $a(x,\xi) = \cF_{\bR^n}( K(x, x- \cdot))$
and that $K$ is smooth away from the diagonal $x=y$. 

\begin{lemma}
\label{lem_tr_Psim<-n}
For any operator $A\in \Psi^m(\Omega)$  with $m<-n$,
 its integral kernel $K$ is continuous on $\Omega\times\Omega$.
 If in addition the symbol $a$ of $A$ is compactly $x$-supported, then  $A$  is trace-class with trace:
$$
\tr (A) = \int_{\Omega} K(x,x) dx 
 =\int_{\Omega\times\bR^{n}} a(x,\xi) \ dx d\xi.
$$
\end{lemma}

The class of smoothing symbols  $S^{-\infty}_{loc}(\Omega\times \bR^n)=\cap_{m\in \bR} S^m_{loc}(\Omega\times \bR^n)$ is equipped with the projective limit topology. 
They can be characterised as follows:
\begin{lemma}
\label{lem_smoothing}
Let $a\in \cup _{m\in \bR}S^m_{loc}(\Omega\times \bR^n)$.
Then $a\in S^{-\infty}_{loc}(\Omega\times \bR^n)$ if and only if the integral kernel $K_a$ of $\Op(a)$ is smooth on $\Omega\times\Omega$.
Furthermore, the mapping $a\mapsto K_a$ is an isomorphism between the topological spaces $S^{-\infty}_{loc}(\Omega\times \bR^n)$ and 
$C^\infty(\Omega\times \Omega)$.
\end{lemma}

We say that a symbol $a\in S^m_{loc}(\Omega\times \bR^n)$ admits an expansion or a poly-homogeneous expansion when $\chi a \in S^m(\bR^n\times \bR^n)$ does in the sense of Section \ref{subsec_exp}
for any $\chi\in C_c^\infty(\Omega)$; we then write 
$a\sim \sum_{j\in \bN_0} \alpha_{m-j}$ and $a\sim_h \sum_j a_{\tilde m-j} $.

\medskip

The functional properties of operators in $\cup_{m\in \bR} \Psi^m(\Omega)$ do not allow for the composition of the operators. 
To obtain the property of composition, one has to consider pseudo-differential operators  which are properly supported, that is, when both $\Omega$-projections of the support of the integral kernel are proper. 
The space $\Psi_{ps}^m(\Omega)$   of properly supported operators in $\Psi^m(\Omega)$ has the following properties:

\begin{proposition}
\label{prop_Psips}
	Any operator in $\Psi_{ps}^m(\Omega)$ 
 sends continuously $C_c^\infty(\Omega)$ to itself, $C^\infty(\Omega)$ to itself, $\cD'(\Omega)$ to itself;
 it also defines a continuous linear operator from $H^s_{comp}(\Omega)$ to $H^{s-m}_{comp}(\Omega)$ 
	and from $H^s_{loc}(\Omega)$ to $H^{s-m}_{loc}(\Omega)$. 
Furthermore, for any $m_1,m_2\in \bR$, the composition $(A,B)\mapsto AB$ is a continuous bilinear mapping $\Psi_{ps}^{m_1}(\Omega)\times \Psi_{ps}^{m_2}(\Omega) \to \Psi_{ps}^{m_1+m_2}(\Omega)$.
\end{proposition}

Any operator in $\Psi^m(\Omega)$ may be written as a properly supported one up to a smoothing operator:

\begin{proposition}
\label{prop_OpPhi}	
	We fix a locally finite partition of unity 
$\Phi=(\phi_k)_{k\in \bN_0}\subset C_c^\infty(\Omega)$, that is,
$1_\Omega=\sum_{k=0}^{+\infty} \phi_k$, 
and for each $k\in \bN_0$, 
 the  set $J_k$ of indices $k'$ such that the supports of $\phi_k$ and $\phi_{k'}$ have a non-empty intersection is finite.

Let $m\in \bR$. For each symbol $a\in S^m_{loc}(\Omega\times\bR^n)$, 
the operator $\Op_\Phi(a) \in \Psi^m_{ps}(\Omega)$ defined via
$$
\Op_\Phi ( a) := \sum_{k=0}^\infty \sum_{k'\in J_k}
\phi_k\Op(a)\phi_{k'},
$$
is in $\Psi^m_{ps}(\Omega)$
with $\Op(a) - \Op_\Phi(a) \in \Psi^{-\infty}(\Omega)$.
Furthermore, 
the linear map $a\mapsto \Op_\Phi (a)$ is continuous from $S^m_{loc}(\Omega\times \bR^n)$ to $\Psi^m_{ps}(\Omega)$, and
the linear map $a\mapsto \Op(a) - \Op_\Phi(a)$ is continuous from $S^m_{loc}(\Omega\times \bR^n)$ to $\Psi^{-\infty}(\bR^n)$.

Moreover, for any $m_1,m_2\in \bR$, the bilinear map $(a,b)\mapsto \Op_\Phi (a)\Op_\Phi(b) - \Op_\Phi(ab)$ is continuous from $S^{m_1}_{loc}(\Omega\times \bR^n)\times S^{m_2}_{loc}(\Omega\times \bR^n)$ to $\Psi_{ps}^{m_1+m_2-1}(\Omega)$.
More generally, 
the bilinear map 
$$
(a,b)\mapsto \Op_\Phi (a)\Op_\Phi(b) - 
\sum_{|\alpha|<N}\frac {(2i\pi)^{-|\alpha|}}{\alpha !}  \Op_\Phi\left(
 \partial_\xi^\alpha a \ \partial_x^\alpha b\right),
 $$
is continuous from $S^{m_1}_{loc}(\Omega\times \bR^n)\times S^{m_2}_{loc}(\Omega\times \bR^n)$ to $\Psi_{ps}^{m_1+m_2-N}(\Omega)$.
\end{proposition}
	
The space $\Psi^{m}_{cl}(\Omega)$ of classical pseudo-differential operators 
of order $m\in \bC$ on $\Omega$ 
is the space of pseudo-differential operators $A\in \Psi^{\Re m}_{ps}(\Omega)$
whose symbols admit  poly-homogeneous expansions.

\subsubsection{The pseudo-differential calculus on a manifold}
Let $M$ be a smooth compact connected manifold of dimension $n$ without boundary.
The space $\Psi^{m}(M)$ of pseudo-differential operators 
of order $m$ on $M$ is the space of operators which are locally transformed by some (and then any) coordinate cover to pseudo-differential operators in $\Psi^{m}(\bR^{n})$;
that is, the operator $A:\cD(M)\to \cD'(M)$ such that there exists a finite open cover $(\Omega_{j})_{j}$ of $M$,
a subordinate partition of unity $(\chi_{j})_{j}$
and diffeomorpshims $F_{j}:\Omega_{j}\to \cO_{j}\subset \bR^{n}$
that transform the operators $\chi_{k }A\chi_{j}$ into operators in $\Psi^{m}(\bR^n)$.

The space $\Psi^{m}_{cl}(M)$ of classical pseudo-differential operators 
of order $m$ on $M$ 
is the space of pseudo-differential operator $A:\cD(M)\to \cD'(M)$
such that each of the operators $\chi_{k }A\chi_{j}:\cD(\Omega_{j})\to \cD'(\Omega_{k})$  in $\Psi^m(\bR^n)$ as above admits a poly-homogeneous expansion.

\subsection{Definition of the non-commutative residue  via local symbols}
\label{subsec_def_res}

Here, we recall the definition of the non-commutative residue via local symbols.
The original references are \cite{guillemin} and \cite{Wod_84,Wod_87}.
See also \cite{Fed++Schrohe,Grubb+Schrohe,schrohe,scott}.

Let $\Omega$ be an open subset in $\bR^n$ with $n\geq 2$.
Let $A\in \Psi^m_{cl} (\Omega)$ with symbols $a\sim_h \sum_{j\in \bN_0} a_{m_j}$.
If $m$ is an integer in 
$$
\bZ_n:=\{-n,-n+1,-n+2,\ldots\},
$$
then we set
$$
\res_x(A) := \int_{\bS^{n-1}} a_{-n}(x,\xi) \ d\sig(\xi);
$$ 
in this paper, $\sig$ denotes the surface measure on the Euclidean unit sphere $\bS^{n-1} \subset \bR^n$ which may be obtained as the restriction to $\bS^{n-1}$ of the $(n-1)$-form $\sig$ defined on $\bR^n$ by
$\sum_{j=1}^n (-1)^{j+1} \xi_j 
\ d\xi_1 \wedge \ldots \wedge d\xi_{j-1} \wedge d\xi_{j+1} \wedge \ldots \wedge d\xi_n$.
If $m\in \bC\backslash \bZ_n$, then we set $\res_x(A):=0$.

\begin{lemma}
\label{lem_intS=0}
Let $b\in C^\infty(\bR^n\backslash\{0\})$ be homogeneous of degree $m\in \bC$
and let $\alpha\in \bN_0^n$.
If $\alpha\not=0$ and $m-|\alpha|=-n$,
then 
$$
\int_{\bS^{n-1}} \partial^\alpha b(\xi) d\varsigma(\xi)=0.
$$
\end{lemma}

Lemma \ref{lem_intS=0}
 may be proved using the fact that
the non-commutative residue is  a trace on $\cup_{m\in \bZ}\Psi^m_{cl}(\Omega)$
\cite{Wod_87},
hence vanishes on commutators.
Indeed, if $a\in S^m_{cl}(\Omega)$,
as $[\Op_\Omega(a),x_j]=\Op_\Omega(\partial_{\xi_j} a)$,
the non-commutative residue of $\Op_\Omega(\partial_j a)$ is zero and so is
the integral over $\bS^{n-1}$ of $(\partial_j a)_{-n}$.
However, we prefer giving here an elementary argument.

\begin{proof}[Proof of Lemma \ref{lem_intS=0}]
Without loss of generality, we may assume  $\partial^\alpha=\partial_1$.
Let $\chi\in \cD(0,\infty)$ be a non-negative function such that $\int_0^\infty\chi(r) dr/r=1$.
On the one hand, 
for any $a\in C^\infty(\bR^n\backslash\{0\})$ homogeneous of degree $-n$,
a polar change of coordinate implies:
$$
\int_{\bR^n} a(\xi) \chi(|\xi|) d\xi
=
\int_0^\infty\int_{\bS^{n-1}} a(\xi)
 \chi(r) d\varsigma(\xi)\frac {dr}r
 =
 \int_{\bS^{n-1}} a(\xi)d\varsigma(\xi).
$$
On the other hand, we have for $a=\partial_1 b$ starting with an integration by parts and then performing a polar change of coordinate:
$$
\int_{\bR^n} \partial_1 b(\xi) \chi(|\xi|) d\xi
=
-\int_{\bS^{n-1}} b(\xi) \frac{\xi_1}{|\xi|} d\varsigma(\xi)
\int_0^\infty \chi'(r) dr.
$$
We conclude with $\int_0^\infty \chi'(r) dr=0$.
\end{proof}

With Lemma \ref{lem_intS=0}, one checks readily that 
if $\tau:\Omega_1\to \Omega_2$ is (smooth) diffeomorphism between two  open sets $\Omega_1,\Omega_2\subset \bR^n$
and if $A\in \Psi^{m}_{cl}(\Omega_1)$, 
then
$$
|\tau'(x)|\res_{\tau(x)} (\tau^* A) =  \res_{x}(A).
$$ 
Hence $\res_x$  yields a 1-density on a compact manifold $M$, for which we keep the same notation $\res_x$.

\begin{definition}
\label{def_residue_density}
The function $x\mapsto \res_x A$ is the \emph{residue density} 
on an open subset $\Omega\subset\bR^n$ or on an $n$-dimensional compact manifold $M$ with $n\geq 2$.
The corresponding integral on $M$ 
$$
\res(A):=\int_M \res_x(A)
$$
or on $\Omega$ when $\res_x(A)$ is integrable
$$
\res(A):=\int_\Omega \res_x(A) \ dx, 
$$
is called  the \emph{non-commutative residue} of $A$.
\end{definition}

The non-commutative residue is a trace on $\cup_{m\in \bC} \Psi^m_{cl}(M)$
in the sense that it is a linear functional on $\cup_{m\in \bC} \Psi^m_{cl}(M)$
which vanishes on commutators.
If $M$ is connected, then any other trace on $\cup_{m\in \bC} \Psi^m_{cl}$ is a multiple of $\res$.

\subsection{Description via heat or power expansions}
\label{subsec_heatexp}

In this section, we recall the expansions of kernels and traces of pseudo-differential operators due to Grubb and Seeley, and their relations with the non-commutative residue.

Recall that a complex sector is a subset of $\bC\backslash\{0\}$ of the form 
$\Gamma=\Gamma_I:=\{re^{i\theta} \ :\  r>0,\ \theta\in I\}$
where $I$ is a subset of $[0,2\pi]$;
it is closed (in $\bC\backslash\{0\}$) when $I$ is closed.	

\begin{theorem}{\cite[Theorem 2.7]{Grubb+Seeley}}
\label{thm_GS}
Let $M$ be a compact smooth manifold of dimension $n\geq 2$
or let $\Omega$ be a bounded open subset in $\bR^n$.
Let $\cL\in \Psi_{cl}$ be an invertible elliptic operator of order $m_0\in \bN$. 
We assume that there exists a complex sector $\Gamma$ such that  the homogeneous principal symbol of $\cL$ in local coordinates satisfies  
$$
\ell_{m_0}(x,\xi)\notin -\Gamma^{m_0} =\{-\mu^{m_0} : \mu \in \Gamma\}
\quad \mbox{when}\ |\xi|=1.
$$

Let $A\in \Psi_{cl}^m$ and let $k\in \bN$ such that $-k m_0 +m<-n$.
The kernel $K(x,y,\lambda)$ of 	$A(\cL-\lambda)^{-k}$ is continuous and satisfies on the diagonal
\begin{equation}
\label{eq1_thm_GS}	
K(x,x,\lambda)\sim 
\sum_{j=0}^\infty c_j(x) \lambda^{\frac{n+m-j}{m_0}-k}
+ 
\sum_{l=0}^\infty \left(c'_l(x)\log \lambda +c''_l(x)\right) \lambda^{-l-k},
\end{equation}
for $\lambda \in -\Gamma^{m_0}$,
$|\lambda|\to \infty$, uniformly in closed sub-sectors of $\Gamma$.
The coefficients $c_j(x)$ and $c'_l(x)$ are determined from the symbols 
$a\sim\sum_j a_{m-j}$ and $\ell\sim \sum_j \ell_{m_0-j}$ in local coordinates, while the coefficients $c''_l(x)$ are in general globally determined.

As a consequence, one has for the trace
\begin{equation}
\label{eq2_thm_GS}	
\tr \left(A(\cL-\lambda)^{-k}\right)
\sim 
\sum_{j=0}^\infty c_j \lambda^{\frac{n+m-j}{m_0}-k}
+ 
\sum_{l=0}^\infty \left(c'_l\log \lambda +c''_l\right) \lambda^{-l-k},
\end{equation}
where the coefficients are the integrals over $M$ of the traces of the coefficients defined in \eqref{eq1_thm_GS}.
\end{theorem}

A closer inspection of the first coefficient in the expansions above shows that they are in relation with the non-commutative residue, see  \cite[Section 1]{schrohe}:

\begin{corollary}
We continue with the setting and results of Theorem \ref{thm_GS}.
The coefficients $c'_0(x)$ and $c'_0$ in \eqref{eq1_thm_GS} and \eqref{eq2_thm_GS} satisfy
$$
c'_0 (x) =  \frac {(-1)^k}{(2\pi)^{n}m_0}\res_x (A) 
\qquad \mbox{and}\qquad
c'_0 =  \frac {(-1)^k}{(2\pi)^{n}m_0}\res (A).
$$
\end{corollary}

Integrating the expansion \eqref{eq2_thm_GS} against  $\lambda^z$ or against $e^{-z\lambda}$ on well chosen $z$-contours yields the following expansions for operators $A$ of any order, see also
 \cite[Section 1]{schrohe}:

\begin{theorem}
\label{thm_schrohe}
Let $\cL\in \Psi_{cl}$ be as in Theorem \ref{thm_GS}.

For any $A\in \Psi_{cl}^m$, we have for $t\to 0$:
\begin{equation}
\label{eq_heat}	
\tr \left(Ae^{-t \cL}\right)
\sim 
\sum_{j=0}^\infty \tilde c_j t^{\frac{n+m-j}{m_0}}
+ 
\sum_{l=0}^\infty \left(\tilde c'_l\ln t  +\tilde c''_l\right) t^{l},
\end{equation}
and 
\begin{equation}
\label{eq_power}	
\Gamma(t)\tr \left(A\cL^{-t}\right)
\sim 
\sum_{j=0}^\infty \frac{\tilde c_j}{s+\frac{n+m-j}{m_0}} t^{\frac{n+m-j}{m_0}}
+ 
\sum_{l=0}^\infty \left(\frac{-\tilde c'_l}{(t+l)^2} +
\frac{\tilde c''_l}{t+l}\right) .
\end{equation}
In \eqref{eq_power}, the left hand side is meromorphic with poles as indicated by the right hand side.
The coefficients $\tilde c_j$, $\tilde c'_l$ and $\tilde c''_l$ 
are multiples of the corresponding $c_j$, $c'_l$ and $ c''_l$
in  \eqref{eq2_thm_GS}.
In particular, 
$$
\tilde c'_0 =  \frac {-1}{(2\pi)^{n}m_0}\res (A).
$$
\end{theorem}

The argument of this paper is to show trace expansions of $A\eta(t\cL)$ for certain functions $\eta$. 
Although our result does not cover the case of  full expansions for
$\eta(\lambda)=e^{-\lambda}$ or $\eta(\lambda)=\lambda^{-t}$ given in Theorem \ref{thm_schrohe},
we will consider a large class of Schwartz functions.
The setting will be slightly different as we do not require $\cL$ to be invertible or $m_0\in \bN$ for instance.

\subsection{Settings and main trace expansions}
\label{subsec_statement_res}

Here we complete the first part of the main result given in the introduction (Theorem \ref{thm_intro}).
We consider two settings: a geometric one on a compact manifold, and a Euclidean one in an open subset of $\bR^n$.

\smallskip

The manifold setting is as follows:
\begin{setting}
\label{set_cL_M}
Let $M$ be a compact manifold of dimension $n\geq 2$.
We consider  an operator   $\cL \in \Psi^{m_0}_{cl}(M)$ with positive order $m_0>0$ satisfying the following hypotheses:
\begin{enumerate}
\item We fix a smooth density on $M$ for which the operator $\cL$ is   formally self-adjoint on $L^2(M)$: 
$$
\mbox{i.e.}\qquad \forall u,v\in C^\infty(M)\qquad
(\cL u,v)_{L^2(\Omega)} = (u,\cL v)_{L^2(M)}.
$$
\item\label{eq_hyp_ellipticM} 
The principal symbol  satisfies the elliptic condition,
$$
\mbox{i.e.}\qquad 
\exists c_0>0\qquad \forall (x,\xi)\in \bT^*M\qquad
\ell_{m_0}(x,\xi)\geq c_0 |\xi|^{m_0},
$$
for one and then any Riemannian structure on $M$.
\end{enumerate}
\end{setting}

Setting \ref{set_cL_M}
together with G\r arding's inequality
imply that  $\cL$ is bounded below and  admits a  unique self-adjoint extension to $L^2(M)$.
Keeping the same notation for this self-adjoint extension, 
its spectrum is discrete and included in $[-c_\cL,+\infty)$, and we have
$$
\cL\geq -c_{\cL},
$$
for some constant $c_{\cL}\in \bR$.

If $M$ is a Riemannian manifold, the most natural choice of $\cL$ is the associated Laplace operator which is a differential operator of order $m_0=2$.

\smallskip

The Euclidean setting is as follows:
\begin{setting}
\label{set_cL_Rn}
We consider an open set $\Omega$ and  an operator   $\cL \in \Psi^{m_0}_{ps}(\Omega)$ with $m_0>0$
 satisfying the following hypotheses:
\begin{enumerate}
\item\label{cond_fsa}  The operator $\cL$ is   formally self-adjoint on $C_c^\infty(\Omega)$, 
$$
\mbox{i.e.}\qquad \forall u,v\in C_c^\infty(\Omega)\qquad
(\cL u,v)_{L^2} = (u,\cL v)_{L^2}.
$$
\item\label{cond_hyp_elliptic} 
The  symbol $\ell$ of $\cL$ admits a homogeneous expansion $\ell\sim_h \sum_{j\geq 0} \ell_{m_0-j}$ and 
 its principal symbol  satisfies the following elliptic condition on any compact subset $K$ of $\Omega$
$$
\mbox{i.e.}\qquad 
\exists c=C_K>0\qquad \forall (x,\xi)\in K \times \bR^n\qquad
\ell_{m_0}(x,\xi)\geq c |\xi|^{m_0}.
$$
\end{enumerate}
\end{setting}

For the Euclidean setting, we will consider the functional calculus for $\cL$ developed in Section \ref{sec_FC}.

\smallskip

The main result of this paper is the following theorem:
\begin{theorem}
\label{thm_main}
We consider one of the following two settings:
\begin{itemize}
\item 	Let $A\in \Psi^{m}_{cl}(M)$ with $m\in \bC$, 
where $M$ is a compact manifold and $\cL\in \Psi^{m_0}_{cl}(M)$ elliptic an operator as in Setting \ref{set_cL_M}.
\item 	Let $A\in \Psi^{m}_{cl}(\Omega)$ with $m\in \bC$ 
whose symbol is compactly $x$-supported in $\Omega$, 
where $\Omega$ is an open set of $\bR^n$ and $\cL\in \Psi^{m}_{cl}(\Omega)$ an elliptic operator as in Setting \ref{set_cL_Rn}.
\end{itemize}

Let  $\eta\in C_c^\infty(0,\infty)$.
The operator 
$A\ \eta(t \cL)$ is trace-class for all $t\in \bR$
and for all $N\in \bN$, there exists a constant $C>0$  such that for all $t\in (0,1)$, we have
$$
|\tr \left(A \eta(t\cL)\right)
-\sum_{j=0}^{N-1} c_{m+n-j} t^{-\frac{m+n-j}{m_0}}|
\leq C t^{\frac{-m-n+N}{m_0}}.
$$
The constant $C$ may be taken of the form 
$C=C_1 \sup_{j=0,\ldots,N_1} \|\eta^{(j)}\|_{L^\infty}$
for some constant $C_1\geq 0$ and some integer $N_1\in \bN_0$ depending on $N$, $\supp(\eta)$ and the setting.  

The constants $c_{m+n-j} = c_{m+n-j}(A,\eta)$ depends on $\eta$ and the setting.
More precisely, they are of the form 
$$
c_{m+n-j'}(A,\eta)=
\tilde c_{m+n-j'}^{(\eta)}
\tilde c_{m+n-j'}^{(A)},
$$
where $\tilde c_{m+n-j'}^{(A)}$ is linear in $A$ and depends on  the setting
and where 
$$
\tilde c_{m+n-j'}^{(\eta)}:=
\frac 1 {m_0}\int_{u=0}^{+\infty}
\eta(u) \ u^{\frac{m-j'+n}{m_0}} \frac{du}u .
$$

If in addition $m\in \bZ_n$ then 
$$
c_0= \ \res(A) \ \int_0^{+\infty} \!\!\! \eta(u)\frac{du}u.
$$

In the $\Omega$-setting, the constants $\tilde c_{m+n-j'}^{(A)}$ depends on the poly-homogeneous expansion of $A$ and $\cL$.
\end{theorem}

Theorem \ref{thm_main} will be proved in Section \ref{subsec_pf_thm_main}.
Our proof will not show that the constants $c_{m+n-j}$ and $c'_{m+n-j}$ above are in relations 
with the constants in the expansions recalled in Section \ref{subsec_heatexp}.
The meaning of $\tr (A\eta(\cL))$ will be discussed in Section \ref{subsec_term1}.

Compared with the known trace expansion recalled in Section \ref{subsec_heatexp}, 
we obtain a result for families of functions $\eta$, instead of the two functions $\eta(\lambda)=e^{-\lambda}$ and $\eta(\lambda)=\lambda^{-z }$, although these two functions do not satisfy the hypotheses above.
The hypotheses on the functions $\eta$ are mainly on its support and boundedness
(beside regularity), thereby excluding holomorphic functions.
The hypotheses on the elliptic operator $\cL$ differ slightly from 
the ones in \cite{Grubb+Seeley,schrohe} recalled in Section \ref{subsec_heatexp} as
we do not require $\cL$ to have an integer order.
Although we do not require $\cL$ to be invertible, 
 the elliptic condition on the principal symbol together with the essentially self-adjointness implies that it is bounded below $\cL\geq -c_\cL$, and therefore $\id+c_\cL+\cL$ is invertible.

\subsection{Some corollaries}
\label{subsec_cor}
We will see that the arguments used in the proof of Theorem \ref{thm_main} lead to a similar expansion 
when the smooth function $\eta$ is not necessarily compactly supported but in the following spaces of functions:

\begin{definition}
\label{def_cMm}
 For any $m\in \bR$, we denote by $\cM^m(\bR)$  the Fr\'echet space of functions $f\in C^\infty(\bR)$
 satisfying 
 $$
 \|f\|_{\cM^m,N}=\sup_{j=0,\ldots, N}
 \sup_{\lambda\in \bR}\ \langle \lambda \rangle^{j-m}
 |f^{(j)}(\lambda)|<\infty, \qquad N=0,1,\ldots
 $$
\end{definition}
Note that $C_c^\infty(\bR)$ is dense in each $\cM^m(\bR)$, $m<0$.

Unfortunately, in the case of $\cM^m$, the expansion we will obtain with these arguments is finite, and 
the estimate for the remaining term goes to $+\infty$ as $t\to 0$, see Part (1) below;
this result will be superseded in Section \ref{subsec_TR}.
We also observe that the result in Theorem \ref{thm_main} classically implies a similar result for any $\eta\in \cM^{m_\eta}$ as long as $\eta$ is supported away from 0, see Part (2) below:

\begin{corollary}
\label{cor_thm_main}
We continue with one of the two settings of Theorem \ref{thm_main}.
Let $\eta\in \cM^{m_\eta}(\bR)$
with $\Re m+m_\eta m_0<-n$.
Then the operator 
$A\ \eta(t \cL)$  is trace-class for all $t\in \bR$.
\begin{enumerate}
\item 
If  $\Re m\geq -n$ (so in particular $m_\eta<0$),
then we have
$$
\forall t\in (0,1)\qquad 
|\tr \left(A \eta(t\cL)\right)
-\sum_{j=0}^{N-1} c_{m+n-j} t^{-\frac{m+n-j}{m_0}}|
\leq C_1 t^{m_\eta},
$$
where $N\in \bN$ denotes the smallest non-negative integer such that $N\geq m_0m_\eta+\Re m+n$.
\item 
If $\supp(\eta)\subset [1,+\infty)$
and $\int_{u=0}^{+\infty}
|\eta(u)| \ u^{\frac{m+n}{m_0}} \frac{du}u<\infty$
 then 
we have for all $N\in \bN$:
$$
\forall t\in (0,1)
\qquad
|\tr \left(A \eta(t\cL)\right)
-\sum_{j=0}^{N-1} c_{m+n-j} t^{-\frac{m+n-j}{m_0}}|
\leq C_2 t^{\frac{-m-n+N}{m_0}}.
$$
\end{enumerate}
In both parts, 
the constants $c_{m+n-j}$ are the same as in Theorem \ref{thm_main}, 
and 
the constants $C_1,C_2$ may be taken of the form $C_j=C'_j \|\eta\|_{\cM^{m_\eta},N_1}$
for some constant $C'_j\geq 0$ and some integer $N_1\in \bN$ depending on  $N$ and the setting.
\end{corollary}

Corollary \ref{cor_thm_main} Part (1) is a generalisation of \cite[Section 2.2]{Li+Stroh} and does not require any hypothesis on the support of the Fourier transform $\widehat \eta$ of $\eta$. Our proof does not use  the ideas and results of \cite{hormander2} or of \cite{duistermaat+guillemin}.
This result will be superseded by Theorem \ref{thm_TR} 
in the next section.
We have chosen to include it as it follows the same type of arguments as for Theorem \ref{thm_main} and prepare for the proof of Theorem \ref{thm_TR}. 

Note that 
Theorem \ref{thm_main} and Part(2) of Corollary \ref{cor_thm_main} imply Part (1) of Theorem \ref{thm_intro}.

\medskip

We will also see that the arguments in the proof of Corollary \ref{cor_thm_main} allow for the following generalisation of Theorem \ref{thm_main}:

\begin{corollary}
\label{cor_thm_main_taylor0}
We continue with one of the two settings of Theorem \ref{thm_main}.
\begin{enumerate}
\item 
The property in Theorem \ref{thm_main} holds 
for any  $\eta\in C_c^\infty(\bR)$
satisfying $\eta^{j}(0)=0$ for all $j=0,1,2,\ldots$.
\item 
For any $N\in \bN$ there exists an integer $N'\in \bN$ such that for all $\eta\in C_c^\infty(\bR)$
satisfying $\eta^{j}(0)=0$ for all $j=0,1,\ldots,N_1$, 
the estimate in Theorem \ref{thm_main} holds.
\end{enumerate}
 \end{corollary}

\subsection{Trace expansion and the canonical trace}
\label{subsec_TR}

In this section, we discuss trace expansions in relations with the canonical trace.
We start with recalling the definition of the canonical trace; references include the original paper  \cite{KV} by Kontsevich and Vishik, as well as \cite{Grubb+Schrohe,paycha}.

\subsubsection{Definition and known properties}
\label{subsubsec_def_TR}

 Let $\Omega$ be a bounded subset of $\bR^n$
and let $A=\Op(a)\in \Psi_{cl}^m(\Omega)$ with complex order $m\notin \bZ_n$.
The symbol of $A$ admits the poly-homogeneous expansion $a\sim_h \sum_{j\in \bN_0} a_{m-j}$.
For $x$ fixed, the function $a_{m-j}(x,\cdot)$ is  smooth  on $\bR^n\backslash\{0\}$ and  $(m-j)$-homogeneous with $m-j\notin \bZ_n$,
so \cite[Theorem 3.2.3]{hormander1} it extends uniquely into a tempered $(m-j)$-homogeneous distributions on $\bR^n$  for which we keep the same notation.
For each $x\in \Omega$, we define the tempered distributions
using the inverse Fourier transform 
$$
\kappa_{a,x}=\cF^{-1} \left\{a(x,\cdot)\right\}
\quad\mbox{and}\quad
\kappa_{a_{m-j},x}=\cF^{-1} \left\{a_{m-j}(x,\cdot)\right\}, 
\ j=0,1,2,\ldots
$$
The distribution $\kappa_{a_{m-j},x}$ is $(-n-m+j)$-homogeneous.
Then for any positive integer $N$ with $m-N<-n$ and $x\in \Omega$
the distribution $\kappa_{a,x} - \sum_{j=0}^{N}	\kappa_{a_{m-j},x}$
is a continuous function on $\bR^n$.
Furthermore, the function $(x,y)\mapsto \kappa_{a,x}(y) - \sum_{j=0}^{N}	\kappa_{a_{m-j},x}(y)$
is continuous and bounded on $\Omega\times \bR^n$.
Its restriction to $y=0$ is independent of $N>m+n$ and defines the quantity
$$
\TR_x(A) := \kappa_{a,x}(0) - \sum_{j=0}^{N}	\kappa_{a_{m-j},x}(0).
$$

If $\tau:\Omega_1\to \Omega_2$ is diffeomorphism between two  open sets $\Omega_1,\Omega_2\subset \bR^n$
and if $A\in \Psi^{\tilde m}_{cl}(\Omega_1)$, 
then
$$
|\tau'(x)|\TR_{\tau(x)} (\tau^* A) =  \TR_{x}(A).
$$ 
Hence, $\TR_x$ yields a 1-density on a compact manifold $M$
for which we keep the same notation $\TR_x$.

\begin{definition}[\cite{KV}]
The density 	$x\mapsto\TR_x(A)$ on a compact manifold $M$ or
$x\mapsto\TR_x(A) dx$ on
 a open subset $\Omega$ is called the \emph{canonical trace density} of $A$.
The corresponding integral on $M$
$$
\TR(A):=\int_M \TR_x(A)
$$
or on $\Omega$ when $\TR_x(A)$ is integrable
$$
\TR(A):=\int_\Omega \TR_x(A) dx,
$$
is called the \emph{canonical trace} of $A$.
\end{definition}

The map $A\mapsto \TR(A)$ is a linear functional  on 
$\Psi^m_{cl}(M)$ for each $m\in \bC\backslash \bZ_n$
and it coincides with the usual $L^2$-trace if $\Re m<-n$.
It is a trace type functional on 
$\cup_{m\in \backslash \bZ_n}\Psi^m_{cl}(M)$ 
in the sense that 
$$
\TR(cA+dB) = c\TR(A) +d\TR(B) 
\qquad \mbox{whenever}\quad c,d\in \bC, \
 A,B\in 
 \cup_{m\in \backslash \bZ_n}\Psi^m_{cl}(M),
  $$
 and 
 $$
 \TR(AB)=\TR(BA) 	
	 \qquad \mbox{whenever}\quad
	 AB,\ BA \in \cup_{m\in \backslash \bZ_n}\Psi^m_{cl}(M).
 $$

The canonical trace was originally defined  in \cite{KV}, and may be defined for a slightly larger set of operators \cite{grubb}.
It coincides with the coefficients $ c_0''$ 
in the expansion \eqref{eq1_thm_GS}.
Furthermore, the residue of the canonical trace of a (suitable) holomorphic family of classical pseudo-differential operators is equal to the non-commutative residue.
It can also be read off the zeta function of $A$ 
and in the $\Omega$-setting is related to the finite-part integral of the symbol of $A$ on $\Omega\times \bR^n$, see also \cite{lesch}.
We will recover this last relation below.

\subsubsection{Result}
\label{subsubsec_statement_TR}
 
Our main result regarding the canonical trace is that it appears as the 0-coefficient of $\tr (A\eta(\cL))$ for functions $\eta$ supported near 0. Because of Corollary \ref{cor_thm_main}, it suffices to consider the case of functions $\eta$ which are compactly supported.

\begin{theorem}
\label{thm_TR}
We consider one of the following two settings:
\begin{itemize}
\item 	Let $A\in \Psi^{m}_{cl}(M)$ with $m\in \bC$, 
where $M$ is a compact manifold and $\cL$ an operator as in Setting \ref{set_cL_M}.
\item 	Let $A\in \Psi^{m}_{cl}(\Omega)$ with $m\in \bC$ 
whose symbol is compactly $x$-supported in $\Omega$, 
where $\Omega$ is an open set of $\bR^n$ and $\cL$ an operator as in Setting \ref{set_cL_Rn}.
\end{itemize}
Let  $\eta\in C_c^\infty(\bR)$.
We assume $m\not \in \bZ$ and $\Re m>  -n$ and that at least one of the two following properties is satisfied: 
\begin{itemize}
\item $\eta\equiv 1$ near 0, 
\item or $m+km_0\not \in \bZ$ for all $k\in \bN$.
\end{itemize}
Then, as $t\to 0$,
 $$
 \tr (A\eta (t\cL)) = \eta(0) \ \TR(A)+ \sum_{j=0}^{N-1} c_{m+n-j} t^{\frac{-m-n+j}{m_0}}  +o(1),
 $$
for $N\in \bN$ the smallest positive integer such that 
$N> \Re m+n$.
The constants $c_{m+n-j}$ are the same as in Theorem \ref{thm_main} and Corollary \ref{cor_thm_main}
and are finite for $j=0,\ldots,N-1$.
\end{theorem}

\begin{remark}
If the order $m_0$ of $\cL$ is an integer, the condition 
$m+km_0\not \in \bZ$ for all $k\in \bN$ is automatically satisfied since $m\not \in \bZ$ is already assumed.
This implies 
 Part (2) of Theorem \ref{thm_intro}.

In the general case, 
the proof will show that we can weaken the hypothesis $m+km_0\not \in \bZ$ for all $k\in \bN$ by replacing it with $m+km_0\not \in \bZ$ for all $k=1,\ldots,N'$ with $N'$ corresponding to $N$ as in Part (2) of Corollary \ref{cor_thm_main_taylor0}.
\end{remark}

Let us consider the case of the Laplace operator $\cL=\Delta$ on a bounded open set $\Omega$.
We consider suitable functions $\eta_k\in C_c^\infty(\bR)$ 
approximating the indicator $1_{[0,1]}$ of the interval $[0,1]$, 
for instance satisfying 
 $0\leq \eta_k \leq 1$, $\eta_k\equiv 1$ on $[0, 1]$ and $\eta_k\equiv 0$ outside $[-\frac 1k, 1+\frac 1k]$.
As $\Delta$ is invariant under translation, we compute easily  for $R>1$:
$$
\tr\left( A (1_{[0,1]} -\eta_k)(R^{-2}\Delta))\right)
=\int_{\Omega\times \bR^n} a(x,\xi)
\left( 1_{[0,1]}-\eta_k\right)(R^{-2}|\xi|^2)\
dx d\xi
=
R^{\Re m +n}\cO(\frac 1k).
$$
Applying Theorem \ref{thm_TR}
to $A=\Op(a)\in \Psi^m_{cl}(\Omega)$, the Laplace operator $\cL=\Delta$ and 
each function $\eta_k$, 
we can take as $k\to +\infty$
and readily obtain as $R\to +\infty$
$$
\tr\left(A\ 1_{[0,1]}(R^{-2}\Delta)\right)=
\int_{|\xi|\leq R}\int_{\Omega}  a (x,\xi) \   dx \ d\xi
=
\TR(A)+ \sum_{j=0}^{N-1} c_{m+n-j} R^{m+n-j}  +o(1).
$$
 In other words, $\TR(A)$ coincides with the finite part of the $\xi$-integral of 
$\int_{\Omega}  a (x,\xi) \ dx $.
Theorem \ref{thm_TR}  allows for a similar description on manifolds for which the spectral description of a Laplace operator is known and 
we can pass to the limit in $\tr A\eta_k(R^{-2}\Delta)$ as $k\to \infty$ as above.
Examples include  for instance the torus (this result is already known, see \cite{Paycha+_tams})
and more generally on compact Lie groups (cf. \cite{fischer_locglob}).

\section{The functional calculus of an elliptic operator}
\label{sec_FC}

In this section, we study the functional calculus of an elliptic operator and its continuous inclusion in the H\"ormander calculus.
Although the techniques (especially the use of the Helffer-Sj\"ostrand formular) are well known to experts in this type of analysis,  
we have chosen to include the proofs for the sake of completeness.

\subsection{The Helffer-Sj\"ostrand formula}
\label{subsec_HSformula}

\subsubsection{Almost analytic extensions}

\begin{definition}
\label{def_almost_an_ext}
Let $f\in C^\infty(\bR)$.
An \emph{almost analytic extension} of $f$ is a function $\tilde f\in C^\infty(\bR^2)$ satisfying for all $x\in \bR$
$$
\tilde f(x,0)=f(x)
\qquad\mbox{and}\qquad
\partial_y^N \bar \partial \tilde f (x,y)|_{	y=0}=0
\ \mbox{for}\ N=1,2,\ldots,
$$ 	
where $\bar \partial$ denotes  the operator 
$\bar \partial = \frac 12 (\partial_x+i\partial_y)$ on $\bR^2$.
\end{definition}

The following lemma gives a practical condition for a function to admit an almost analytic extension:

\begin{lemma}
\label{lem_exist_tildef}
Let $f\in C^\infty(\bR)\cap \cS'(\bR)$.
If  $\xi^N \widehat f\in L^1(\bR)$ for every $N\in \bN_0$, then we can construct 
an almost analytic extension of $f$. 
\end{lemma}
\begin{proof}
We fix a functions $\chi\in C_c^\infty(\bR)$
such that $\chi(s)=1$ for $|s|\leq 1$ and $\chi(s)=0$ for $|s|>2$. 
We then define the function $\tilde f:\bR^2\to \bC$  with
$$
\tilde f(x,y) 
:=
\int_\bR e^{2i\pi(x+iy)\xi} \chi(y\xi) \widehat f(\xi) d\xi.
$$
One checks easily that $\tilde f$ is smooth, that $\tilde f(x,0)=f(x)$ for any $x\in \bR$ by the Fourier inverse formula and that we have
$$
\bar \partial \tilde f(x,y) 
:=
{\frac i 2}
\int_\bR e^{2i\pi(x+iy)\xi} 
 \chi'(y\xi) \xi \, \widehat  f(\xi) \, d\xi;
$$
note that the last integrant is  supported in $\{\xi\in \bR : |y|^{-1} \leq |\xi|\leq 2|y|^{-1}\}$ if $y\not=0 $ and
identically zero if $y=0$.
We observe that for all $M\in \bN_0$ and any $y\in\bR\setminus \{0\}$, we have
\begin{equation}
\label{eq_est_barpartialtildef}
|\bar \partial \tilde f(x,y) |
=
\frac 12 
|\int_\bR e^{2i\pi(x+iy)\xi} \frac{\chi'(y\xi)}{(y\xi)^M} (y\xi)^M \xi \widehat f(\xi) d\xi|
\lesssim_{\chi,M} |y|^M \int_\bR|\xi^{M+1} \widehat f(\xi) |d\xi.
\end{equation} 
From this, one sees readily that $\tilde f$ is an almost analytic extension of $f$.
\end{proof}

Lemma \ref{lem_exist_tildef} implies that, for instance, 
 a smooth function $f\in C_c^\infty(\bR)$ with compact support
 admits  an almost analytic extension.
 In fact, it is easy to construct an almost analytic extension 
   $\tilde f_1 \in C_c^\infty(\bR^2)$ with compact support  by setting
$\tilde f_1 (x,y) = \chi_1(x)\chi_2(y) \tilde f (x,y)$	
with $\tilde f$  as in the proof of lemma \ref{lem_exist_tildef}, and  $\chi_1,\chi_2\in C_c^\infty(\bR)$ supported near the support of $f$ and 0 respectively with $\chi_1=1$ on $\supp \ f$ and $\chi_2\equiv 1$ near 0.
Lemma \ref{lem_exist_tildef} and the following statement imply that we can construct analytic extensions of larger classes of smooth functions:
\begin{lemma}
\label{lem_cMm}
Let $f\in \cM^m(\bR)$ with $m<-1$ (see Definition \ref{def_cMm}).
For any $N\in \bN_0$, 
the function $\xi^N \widehat f$ is integrable and 
$$
\|\xi^N \widehat f\|_{L^1} \leq C_N \|f\|_{\cM^m,N+2}.
$$
Moreover, for any $N,M\in \bN_0$, provided that $m+M-N<0$, then  $\xi^N \widehat f^{(M)}$ is integrable and we have
$$
\|\xi^N \widehat f^{(M)}\|_{L^1} \leq C_{N,M} \|f\|_{\cM^m,N+2}.
$$	
Above the constants $C_N$ and $C_{N,M}$ are positive and  depends on $m,N$ or $m,N,M$ but not on $f$. 
\end{lemma}
 
 \begin{proof}
For any $N\in \bN_0$, 
the function $\xi^N \widehat f $ is continuous on $\bR$ and bounded by
$$
\|\xi^N \widehat f \|_{L^\infty} \leq 
(2\pi)^{-N} \|f^{(N)}\|_{L^1}
\lesssim
\|f\|_{\cM^m,N},
$$
therefore it is integrable with
$$
\|\xi^N \widehat f\|_{L^1} 
\leq 
\int_{|\xi|\leq 1} |\widehat f(\xi)|\, d\xi
+
\int_{|\xi|> 1}|\xi^N \widehat f(\xi)|\, d\xi
\lesssim \|\widehat f\|_{L^\infty} +
\|\xi^{N+2} \widehat f\|_{L^\infty} 
\lesssim
 \|f\|_{\cM^m,N+2}.
$$
We conclude the proof by 
applying what has just been proven to $g= (2i\pi)^{M-N} (x^M f)^{(N)} \in \cM^{m+M -N}$.
 \end{proof}

We can refine the construction of almost analytic extension of a function in $\cM^m(\bR)$ given by Lemmata \ref{lem_exist_tildef} and \ref{lem_cMm}:

\begin{lemma}
\label{lem_tildefcMm}
Let $f\in \cM^m(\bR)$ with $m<-1$.
 We can construct an almost analytic extension $\tilde f$ which is supported in $\{(x,y)\in \bR^2 \ : \ |y|\leq 2\}$ and satisfies
 for all integers $N,N'\in \bN_0$ 
$$
\int_{\bR^2} 
|\bar \partial \tilde f(x,y)| \frac {\langle x,y\rangle^N}{|y|^{N'}}
dx dy 
\leq C' \|f\|_{\cM^m,N'+N+5},
$$ 	
where the constant $C,C'>0$ depends on $N,N'$ and on the construction of $\tilde f$ but not on $f$.
\end{lemma}

\begin{proof}
Lemma \ref{lem_cMm} allows us to construct the analytic extension $\tilde f$ for $f$ from the proof of Lemma \ref{lem_exist_tildef}.
We generalise \eqref{eq_est_barpartialtildef} in the following way: 
\begin{align*}
|x^\alpha\bar \partial \tilde f(x,y)| 
&\lesssim 
|y|^{M} \left|\int_\bR e^{2i\pi x\xi}
\partial_\xi^{\alpha}\left\{\frac{e^{-2\pi y\xi}\chi'(y\xi)}{(y\xi)^M} 
\xi^{M+1} \widehat f(\xi)\right\} d\xi\right|
\\&\lesssim  |y|^{M} 
\sum_{0\leq j_1+j_2\leq \alpha} \|\xi^{M+1 -j_1} \widehat f^{(j_2)}\|_{L^1}
\lesssim  |y|^{M}
\|f\|_{\cM^m,M+3},
\end{align*}
for any $M\geq \alpha$ and $|y|\leq 2$
by Lemma \ref{lem_cMm}.
 This estimate applied 
 with $\alpha=0,\ldots,N+2$ and $M =\max(N',N+2)$
yields
the result for the almost analytic extension $\tilde f_1$ of $f$ given by 
$\tilde f_1(x,y)=\chi(y) 
\tilde f(x,y).$
	 \end{proof}

\subsubsection{Formulae}
We may consider  an analytic extension as a function of the complex variable $z=x+iy\in \bC$.
We denote by  $L(dz)=dxdy$ the Lebesgue measure on $\bC$ identified with $\bR^2$.
The Cauchy-Pompeiu integral formula then yields:
\begin{lemma}
\label{lem_cauchyF}
	Let $f\in \cM^m(\bR)$ with $m<-1$.
	For any almost analytic extension $\tilde f$ of $f$ 
	such that  $\int_\bC \bar \partial \tilde f(z) \ L(dz)/ |\Im z|<\infty$, 
we have 
$$
\forall \lambda\in \bR\qquad
f(\lambda)=
\frac 1 {\pi} \int_\bC \bar \partial \tilde f(z) \ (\lambda-z )^{-1}L(dz) .
$$
Furthermore, for any $j=0,1,2,\ldots$, 
we have
$$
\forall \lambda\in \bR\qquad
\frac{(-1)^j}{j!}f^{(j)}(\lambda)=
\frac 1 {\pi} \int_\bC \bar \partial \tilde f(z) \ (\lambda-z )^{-1-j}L(dz) ,
$$
as long as the almost analytic extension $\tilde f$ of $f$ 
	satisfy $\int_\bC \bar \partial \tilde f(z) \ L(dz)/ |\Im z|^{j'}<\infty$ for any $j'\in \bN$.
	Such almost analytic extensions can be constructed (see Lemma~\ref{lem_tildefcMm}).
\end{lemma}

The Helffer-Sj\"ostrand formula has become a fundamental tool in functional analysis of a self-adjoint operator $T$ (densely defined on a separable Hilbert space $\cH$) since it is easily proved that the resolvent satisfies:
\begin{equation}
\label{eq_resolvent}	
\forall z\in \bC \setminus \bR \qquad 
\|(T-z)^{-1} \|_{\sL(\cH)} \leq |\Im z|^{-1}.
\end{equation}
This together with Lemma \ref{lem_cauchyF} yield the Helffer-Sj\"ostrand formula
(references for this include \cite{bouclet,davies_bk,dimassi+Sj,helffer+Sj}):

\begin{theorem}
\label{thm_sp_tildef}
Let  $T$ be a self-adjoint operator densely defined on a separable Hilbert space $\cH$.
For any $f\in \cM^m(\bR)$ with $m<-1$,
the spectrally defined operator $f(T)\in \sL(\cH)$    coincides with  
$$
f(T) = \frac 1{\pi} \int_\bC \bar \partial \tilde f(z)\
(T-z)^{-1} L(dz).
$$
Here,  $\tilde f$ is any almost analytic extension of $f$ such that $\int_{\bC} |\bar \partial \tilde f(z)| |\Im z|^{-1} L(dz)$ is finite;
such extensions exist (see Lemma~\ref{lem_tildefcMm}).
\end{theorem}

\subsection{The resolvant in the Euclidean setting}

In the Euclidean setting, 
we can define the formal resolvent of $\cL$ in the following way:

\begin{lemma}
\label{lem_cLOmega'}
	We consider  Setting \ref{set_cL_Rn}.
	\begin{enumerate}
	\item 
		For any  bounded open set  $\Omega'$  such that $\bar \Omega'\subset \Omega$,
the restriction of $\cL$ to an unbounded operator on $L^2(\Omega')$ densely defined on $C_c^\infty(\Omega')$  admits a unique self-adjoint extension $\cL_{\Omega'}$ to $L^2(\Omega')$. 
Furthermore,  $\cL_{\Omega'}$ is bounded below:
$$
\cL_{\Omega'}\geq -c_{\Omega'},
$$
and the operator $\cL_{\Omega'}$ coincides with the restriction of $\cL$ to $\Omega'$ in the following sense: 
\begin{itemize}
\item 
for any function $v\in L^2(\Omega')$ which we view also as a distribution $v\in \cE'(\Omega)$, 
if the distribution $\cL v \in \cE'(\Omega)$ coincides with a square integrable function on $\Omega'$, 
then $v$ is in the domain $\Dom (\cL_{\Omega'})$ of $\cL_{\Omega'}$ and 
the distribution $\cL v\in \cE'(\Omega)$ coincides with the function $\cL_{\Omega'} v \in L^2(\Omega') $ on $\Omega'$;

\item 
conversely,
for any function $v\in \Dom (\cL_{\Omega'})$, 
$\cL_{\Omega'} v = \cL v|_{\Omega'} \in L^2(\Omega')$.
\end{itemize}

\item 
The operator $\cL$ is   formally self-adjoint on $L^2(\Omega)\cap \cE'(\Omega)$ in the sense that for any $u,v\in L^2(\Omega)\cap \cE'(\Omega)$ such that $\cL u $ and $\cL v$ are both in $L^2(\Omega)$, we have
$$
(\cL u, v)_{L^2(\Omega)} = (u,\cL v)_{L^2(\Omega)}.
$$

\item 
For any $z\in \bC\setminus \bR$, 
for any $u\in C_c^\infty(\Omega)$,
there exists a unique function $v\in L_{loc}^2(\Omega)$ such that 
$u=(\cL-z)v$.
Furthermore, $v=(\cL-z)^{-1} u$ is in $L^2(\Omega)$
with $\|v\|_{L^2(\Omega)} \leq |\Im z|^{-1} \|u\|_{L^2(\Omega)}$.
 It is
smooth and satisfies for any $N\in \bN$
$$
\|(\cL-z)^{-1}u\|_{H^N_{loc}(\Omega),N}
\leq \frac{C}{|\Im z|} \|u\|_{H^{N'}_{loc}(\Omega),N'}
$$
where the integer $N'$ and the constant $C\geq 0$ are independent of $u$ or $z$ - but may depend on $N$ and Setting \ref{set_cL_Rn}.	
\item 
For any $z\in \bC\setminus \bR$, 
the operator $(\cL-z)^{-1}$ maps $C_c^\infty(\Omega)$ to $\cap_{s\in \bR} H^s_{loc}(\Omega)$ linearly and continuously. 
Furthermore, it satisfies:
$$
\forall u\in C_c^\infty(\Omega)\qquad
\|(\cL-z)^{-1} u\|_{L^2(\Omega)} \leq \frac 1{|\Im z|}\| u\|_{L^2(\Omega)} .
$$ 
	\end{enumerate}
\end{lemma}

\begin{proof}
By the G\r arding inequality, $\cL$ is bounded below on $C_c^\infty(\Omega')$.
We denote by $\cL_{\Omega'}$ the Friedrichs self-adjoint extension to $L^2(\Omega')$. 
The construction of a parametrix yields elliptic estimates which shows that the equation $\cL_{\Omega'} u = \pm iu$ for $u\in \Dom (\cL_{\Omega'})$ has only $u=0$ as solution.
Hence the self-adjoint extension to $L^2(\Omega')$ is unique.

Let us consider $v\in L^2(\Omega')$ such that $\cL v|_{\Omega'}\in L^2(\Omega')$.
Let  $(v_j)$ be a sequence in $C_c^\infty (\Omega')$ converging to $v$ in $L^2(\Omega')$.
Hence, $\cL v_j|_{\Omega'} = \cL_{\Omega'} v_j$ defines a sequence in $L^2(\Omega')$ which is bounded and weakly convergent to  $\cL v|_{\Omega'}$.
One then checks that the sequence $(v_j)$ is Cauchy for the Friederichs sesquilinear form of $\cL_{\Omega'}$.
Therefore $v\in \Dom (\cL_{\Omega'})$ and $\cL_{\Omega'} v =\cL v|_{\Omega'}$. 

Let us consider $v \in \Dom (\cL_{\Omega'})$.
There exists a sequence $(v_j)$ in $C_c^\infty (\Omega')$ converging to $v$ in $L^2(\Omega')$ 
and Cauchy for  the Friederichs sesquilinear form of $\cL_{\Omega'}$.
 Hence
$\cL v_j|_{\Omega'} = \cL_{\Omega'} v_j$ defines a (bounded) sequence converging weakly to $\cL_{\Omega'} v$ in $L^2(\Omega')$. As $(\cL v_j)$  converges to $\cL v$ in $\cE'(\Omega)$, we conclude that $\cL_{\Omega'} v = \cL v|_{\Omega'}$.
This shows Part (1). 

\smallskip

For Part (2), consider $u,v\in L^2(\Omega)\cap \cE'(\Omega)$ with $\cL u,\cL v \in L^2(\Omega)$.
We may assume that $u$ and $v$ are both supported in a bounded open set $\Omega_1$ with $\bar \Omega_1 \subset \Omega$.
Let $\Omega_2$ be a bounded open set with $\bar \Omega_2\subset \Omega$  and $\Omega_2 \supset \supp (w)$ 
for any $w\in \cE'(\Omega)$ supported in $\Omega_1$. 
Let $(u_j)$ and $(v_j)$ be sequences in $C_c^\infty(\Omega_1)$ converging to $u$ and $v$ respectively in $L^2(\Omega_2)$.
Then we have $(\cL u_j, v_j)_{L^2(\Omega)} = (u_j, \cL v_j)_{L^2(\Omega)}$, and 
 the limit as $j\to \infty$ yields the proof of
 Part (2).

\smallskip

Now let us prove Part (3).
We consider $z\in \bC \setminus \bR$ and $u\in C_c^\infty(\Omega)$.
For any bounded open set $\Omega'$ with 
$\bar \Omega'\subset \Omega$ and $\Omega'\supset \supp (u)$, 
we set
$$
L^2(\Omega') \ni v_{\Omega'} := (\cL_{\Omega'} -z)^{-1}u.
$$
By \eqref{eq_resolvent}, $\|v_{\Omega'}\|_{L^2(\Omega')} \leq |\Im z|^{-1} \|u\|_{L^2(\Omega')}$.
 If we define $v_{\Omega''}$ in a similar way on another  bounded open subset $\Omega''$ 
 with $\bar \Omega''\subset \Omega$ and $\Omega''\supset \bar \Omega'$, 
then   Part (2) implies
$$
z\|v_{\Omega'}-v_{\Omega''}\|_{L^2(\Omega')}^2
=
z(v_{\Omega'}-v_{\Omega''} , v_{\Omega'}-v_{\Omega''})_{L^2(\Omega')}
=
(\cL (v_{\Omega'}-v_{\Omega''}) , v_{\Omega'}-v_{\Omega''})_{L^2(\Omega')},
$$
and similarly for $\bar z \|v_{\Omega'}-v_{\Omega''}\|_{L^2(\Omega')}^2$, so 
\begin{align*}
2\Im z \|v_{\Omega'}-v_{\Omega''}\|_{L^2(\Omega')}^2
&= 
(\cL (v_{\Omega'}-v_{\Omega''}) , v_{\Omega'}-v_{\Omega''})_{L^2(\Omega')}
-
(v_{\Omega'}-v_{\Omega''}, \cL(v_{\Omega'}-v_{\Omega''}))_{L^2(\Omega')}	\\
&=(\cL v_{\Omega''} , v_{\Omega''}1_{\Omega'})_{L^2(\Omega)}
-
(v_{\Omega''}1_{\Omega'}, \cL v_{\Omega''})_{L^2(\Omega)} = 0;
\end{align*}
here $1_{\Omega'}$ denotes the indicatrix of the set $\Omega'$.
Therefore $v_{\Omega'}=v_{\Omega''}$ on $\Omega'$.
This allows us to define $v\in L^2_{loc}(\Omega)$ via $v=v_{\Omega'}$ 
on any bounded open set $\Omega'$ with 
$\bar \Omega'\subset \Omega$ and $\Omega'\supset \supp (u)$.
It satisfies 
$(\cL-z)v=u$ on $\Omega$ and $\|v\|_{L^2(\Omega)} \leq |\Im z|^{-1} \|u\|_{L^2(\Omega)}$.

Let us show $v\in \cap_{s\in \bN} H^s_{loc}(\Omega)$.
We will need the following observation:
for any bounded open set $\Omega_1$ such that $\bar \Omega_1\subset \Omega$, 
we can construct a parametrix of $(1+c_{\Omega_1} +\cL_{\Omega_1})^M$ and obtain the estimates
\begin{align*}
\|(1+c_{\Omega_1} +\cL_{\Omega_1})^M w \|_{L^2(\Omega_1)}
&\lesssim 
\| (1+\Delta)^{\lceil m_0/2\rceil M} w \|_{L^2(\Omega_1)},
\\
\| (1+\Delta)^{ M} w \|_{L^2(\Omega_1)}
&\lesssim 
\|(1+c_{\Omega_1} +\cL_{\Omega_1})^{\lceil 2/m_0\rceil M} w \|_{L^2(\Omega_1)},
\end{align*}
where the implicit constants do not depend on $w\in C_c^\infty(\Omega_1)$ but may depend on $M\in \bN$, and $\cL,\Omega_1$; here $\Delta$ denotes the standard 
Laplacian in $\bR^n$ and $\lceil x\rceil$ the smallest non-negative integer strictly greater than $x\in \bR$.
This implies for any $\Omega'$ as above:
\begin{align*}
&\| (1+\Delta)^{ M} v \|_{L^2(\Omega')}
=
\| (1+\Delta)^{ M} v_{\Omega'} \|_{L^2(\Omega')}
\lesssim 
\|(1+c_{\Omega'} +\cL_{\Omega'})^{\lceil 2/m_0\rceil M} v_{\Omega'}\|_{L^2(\Omega')}
\\
&\qquad \leq \frac 1 {|\Im z|}
\| (1+c_{\Omega'} +\cL_{\Omega'})^{\lceil 2/m_0\rceil M} u \|_{L^2(\Omega')}
\lesssim
 \frac 1 {|\Im z|}
 \| (1+\Delta)^{\lceil  m_0/2 \rceil  \lceil  2 /m_0\rceil M} u \|_{L^2(\Omega')}.
\end{align*}
This shows  $v\in \cap_{s\in \bN} H^s_{loc}(\Omega)$,  so $v$ is smooth and Part (3) follows.

Part (4) is a consequence of Part (3).
\end{proof}

\subsection{The functional calculus}
\label{subsec_FC}

Here we define and relate the functional calculus of an elliptic operator $\cL$ in Setting \ref{set_cL_Rn}
with the pseudo-differential calculus.
At least formally, we use the same definition as for the  Helffer-Sj\"ostrand formula (see Section \ref{subsec_HSformula}):

\begin{definition}
\label{def_f(cL)u}
We consider Setting \ref{set_cL_Rn}.
For any $f\in \cM^m(\bR)$ with $m<-1$, 
and any $u\in C_c^\infty(\Omega)$, 
we define the distribution  
$$
 f(\cL)u :=  
\frac 1{\pi} \int_\bC \bar \partial \tilde f(z)\
(\cL-z)^{-1}u \  L(dz),
$$
where $\tilde f$ is any almost analytic extension of $f$ such that $\int_{\bC} |\bar \partial \tilde f(z)| |\Im z|^{-1} L(dz)$ is finite
and where the resolvent $(\cL-z)^{-1}$ takes its meaning from Lemma \ref{lem_cLOmega'}.
\end{definition}

Let us fix  an almost analytic extension $\tilde f$ such that 
$$
C_{\tilde f}:=
\int_{\bC} |\bar \partial \tilde f(z)| |\Im z|^{-1} L(dz),
$$ is finite; by  Lemma \ref{lem_tildefcMm}, such an extension exists.
By Lemma \ref{lem_cLOmega'} Part (4), 
	the distribution $f(\cL)u$ is square integrable with the estimate 
	 $$
\|f(\cL)u\|_{L^2(\Omega)} \leq C_{\tilde f} \|u\|_{L^2(\Omega)}.
$$
Furthermore, for any $v\in C_c^\infty(\Omega)$, 
we have 
\begin{align*}
(f(\cL)u,v)_{L^2(\Omega)}
&= 
\frac 1{\pi} \int_\bC \bar \partial \tilde f(z)\
((\cL-z)^{-1}u ,v)_{L^2(\Omega)}\  L(dz)
\\&= 
\frac 1{\pi} \int_\bC \bar \partial \tilde f(z)\
((\cL_{\Omega'}-z)^{-1}u ,v)_{L^2(\Omega')}\  L(dz)
=(f(\cL_{\Omega'} u,v)_{L^2(\Omega')},
\end{align*}	
where $\Omega'$ is a bounded open set containing the supports of $u$ and $v$ and with $\bar \Omega'\subset \Omega$.
As the Helffer-Sj\"ostrand formula for $f(\cL_{\Omega'})$ does not depend on the choice of the almost analytic extension $\tilde f$ with $C_{\tilde  f}<\infty$,
so does $f(\cL)u$ in Definition  \ref{def_f(cL)u}.

\begin{definition}
\label{def_f(cL)}
Definition \ref{def_f(cL)u} yields  the continuous operator  $f(\cL):C_c^\infty(\Omega)\to L^2(\Omega)$.
\end{definition}

Since the Helffer-Sj\"ostrand formula may be used to define the spectral calculus of a self-adjoint operator (see Section \ref{subsec_HSformula}), 
the spectral calculus defined for any self-adjoint for $\bar \cL$ will coincide with our definition: 

\begin{lemma}
\label{lem_f(cL)_f(barcL)}
If $\cL$ admits a self-adjoint extension $\bar \cL$ to $L^2(\Omega)$, then $f(\bar \cL)$ coincides on $C_c^\infty(\Omega)$ with $f(\cL)$ from Definition \ref{def_f(cL)} for any $f\in \cM^m$, $m<-1$. And this is so for any such extension $\bar \cL$.
\end{lemma}

The main result of this section is the following pseudo-differential property of this functional calculus:
\begin{theorem}
\label{thm_f(L0)}
We consider  Setting \ref{set_cL_Rn} and the functional calculus from Definition \ref{def_f(cL)}.

\begin{enumerate}
\item Let $m<-1$ and  $f\in \cM^m(\bR)$.
The operator $f(\cL)$ is in $\Psi^{m m_0}_{ps}(\Omega)$
and its symbol admits an expansion $a\sim \sum_{j\in \bN_0} a_{m m_0-j}$ 
satisfying for $|\xi|\geq 1$ 
$$
a_{m m_0}(x,\xi)=f(\ell_{m_0}(x,\xi)), 
$$
and more generally for any $j=0,1,2,\ldots$
$$
a_{m m_0-j}(x,\xi)=
\sum_{k=0}^j c_{j,k}f^{(j+k)}(\ell_{m_0}).
$$
where  each $c_{j,k}\in S^{-j+(j+k)m_0}_{loc} (\Omega\times \bR^n)$ is a symbol independent of $z$, 
identically 0 for $|\xi| <1$
and
homogeneous of degree $-j+(j+k)m_0$ in $\xi$ for  $|\xi|\geq 2$.
\item 
For each $m<-1$, 
the map
	$f\mapsto f(\cL)$ is continuous
	from $\cM^m(\bR)$  to $\Psi^{mm_0}(\Omega)$.	
	Moreover, the map $f\mapsto f(\cL)-\sum_{j=0}^{N}  \Op(a_{mm_0-j})$  is continuous
	from $\cM^m(\bR)$  to $\Psi_{ps}^{mm_0-N-1}(\Omega)$ for any $N\in \bN_0$.

\item If $f\in \cS(\bR)$ then $f(\cL)\in \Psi^{-\infty}(\Omega)$ is smoothing.	
\end{enumerate}
\end{theorem}

The proof of Theorem \ref{thm_f(L0)} is given in Appendix \ref{sec_pf_thm_f(L0)}. It relies on the Helffer-Sj\"orstand formula and on the construction of a paramterix $P_z$ for $(\cL-z)$.

First let us state some corollaries.

\begin{corollary}
\label{cor_thm_f(L0)}
We consider  Setting \ref{set_cL_Rn} and the functional calculus from Definition \ref{def_f(cL)}.
\begin{enumerate}
\item If $\cL$ admits a self-adjoint extension $\bar \cL$ on $L^2(\Omega)$, then for any $f\in \cM^m(\bR)$ with $m<-1$, 
the operator $f(\bar \cL)$ defined spectrally as a bounded operator on $L^2(\Omega)$ coincides with the pseudo-differential operator $f(\cL)$ on $L^2(\Omega)$.  
\item 
If $\cL$ is bounded below on $C_c^\infty(\Omega)$,
then it admits a unique self-adjoint extension $\bar \cL$.
Furthermore, for any $f\in \cM^m(\bR)$ with $m\in \bR$, the spectrally defined operator $f(\bar \cL)$ coincide with a pseudo-differential operator in $\Psi^{mm_0}_{ps}(\Omega)$, and the properties of Theorem \ref{thm_f(L0)} holds even for $m\geq -1$.
\end{enumerate}
\end{corollary}

\begin{proof}
Part 1 follows from Theorem \ref{thm_f(L0)} and Lemma \ref{lem_f(cL)_f(barcL)}. 
	For Part 2, 
	if $m\geq -1$, 
	consider $N\in \bN$ such that $m-N<-1$ and $f_1(\lambda):=(1+c+\lambda)^{-N} f(\lambda)$ 
	where the constant $c$ is such that $\cL \geq -c$.
	Then $f(\bar \cL) = (1+c+\bar \cL)^N f_1(\bar \cL)$ and the result follows from Part 1 and Theorem \ref{thm_f(L0)}.
\end{proof}

\begin{remark}
\label{rem_thm_f(L0)}
\begin{itemize}
\item
Theorem \ref{thm_f(L0)} and Corollary \ref{cor_thm_f(L0)} extend readily to the case of an operator $\cL$ valued in a finite dimension real vector space,
we will not use this in this paper.
\item 
Part (3) in Theorem \ref{thm_f(L0)} and Lemma \ref{lem_tr_Psim<-n}
imply that the function $t\mapsto \tr (e^{it\cL})$ is the tempered distribution  given by 
$\cS(\bR)\ni \phi\mapsto  \int_\bR \tr (e^{it\cL})\ \phi (t) dt = \tr (\widehat \phi(\cL))$.
\end{itemize}
\end{remark}

Theorem \ref{thm_f(L0)} implies a similar result in the setting of compact manifolds:

\begin{corollary}
\label{cor_M_thm_f(L0)}
We consider  Setting \ref{set_cL_M}.
\begin{enumerate}
\item 
For any function $f\in \cM^m(\bR)$ with $m\in \bR$,
the spectrally defined operator $f(\cL)$ is  in $\Psi^{m m_0}(M)$.
\item 
The map
	$f\mapsto f(\cL)$ is continuous
	from $\cM^m(\bR)$  to $\Psi^{mm_0}(M)$.	
\item If $f\in \cS(\bR)$ then $f(\cL)\in \Psi^{-\infty}(M)$ is smoothing
and the map $\cS(\bR)\ni f \mapsto f(\cL) \in \Psi^{-\infty}(M)$ is continuous.
The integral kernel $K_f$ of $f(\cL)$ is a smooth function on $M \times M$, and the map $\cS(R)\ni f \mapsto K_f\in C^\infty(M\times M)$ is continuous.
	By duality, the map $\cS'(\bR)\ni f \mapsto K_f\in \cD'(M\times M)$ is continuous.
\end{enumerate}
\end{corollary}

As in Remark \ref{rem_thm_f(L0)}, 
Corollary \ref{cor_M_thm_f(L0)} extends to an operator $\cL$ on a vector bundle over $M$.
Part (3) of Corollary \ref{cor_M_thm_f(L0)} allows us to consider the integral kernel of $e^{it \cL}$.

\section{Proof of Theorems \ref{thm_main} and \ref{thm_TR}}
\label{sec_pf_thms}

Here, we prove Theorem \ref{thm_main}, its corollaries, and Theorem  \ref{thm_TR} in Sections \ref{subsec_pf_thm_main},  \ref{subsec_proof_cor}
and \ref{subsec_proof_thm_TR} respectively.
Our first task is to understand the meaning of $\tr (A\eta(\cL))$.

\subsection{Meaning of $\tr (A\eta(\cL))$ and the first term in the expansion}
\label{subsec_term1}
In the manifold case (i.e. Setting~\ref{set_cL_M}), 
 $\cL$ admits a unique self-adjoint extension, 
so the operator $\eta(\cL)\in \sL(L^2(M))$ can be  defined spectrally for any $\eta\in L^\infty(\bR)$.
Hence if  $A\in \Psi^m(M)$ with $m<-n$, then 
$A\ \eta(\cL)$ is trace-class with
$$
\tr \left|A\ \eta(\cL)\right|
\leq
\|\eta(\cL)\|_{\sL(L^2(\Omega))}
\tr |A|
\leq \|\eta\|_{L^\infty} \ \tr |A|.
$$
More generally, if  $\eta:\bR\to \bC$ is a measurable function such that 
\begin{equation}
\label{eq_eta_mes_meta}
\exists C>0\qquad 
\forall \lambda\in \bR\qquad
|\eta(\lambda)|\leq (1+\lambda)^{m_\eta},
\quad\mbox{where}\quad
m_\eta<(-n-m)/m_0,
\end{equation}
 then 
$A\ \eta(\cL)$ is trace-class with 
$\tr \left|A\ \eta(\cL)\right|
\leq C \|A\|_{\Psi^m, N}$ for some $N\in \bN$.

We can generalise this to the Euclidean setting in the following way:

\begin{lemma}
\label{lem_trAetaOmega'}
	We consider Setting \ref{set_cL_Rn}.
Let $A\in \Psi^{m}(\Omega)$ and such that its integral kernel $K_A$ is compactly supported in $\Omega\times \Omega$.
Let $\Omega'$ be a bounded open set such that $\bar \Omega'\subset \Omega$ and $\supp(K_A)\subset \Omega'\times \Omega'$.
Using the notation of Lemma \ref{lem_cLOmega'}, 
the operator $\eta(\cL_{\Omega'})$ may be defined spectrally for any measurable function $\eta:\bR\to \bC$.
\begin{itemize}
\item	
If $m<-n$ and $\eta\in L^\infty(\bR)$, we define $\tr (A\eta(\cL)) :=\tr (A\eta(\cL_{\Omega'}))$.
 \item More generally, if $m$ and $\eta$ satisfy
 \eqref{eq_eta_mes_meta}, then setting for $\epsilon_0>0$
$$
\Psi^{-n-\epsilon_0}_{cl} \ni A_1:=A (1+c_{\Omega'}+\cL_{\Omega'})^{-\frac{m +n+\epsilon_0}{m_0}} 
 \qquad\mbox{and}\qquad
 \eta_1 (\lambda) := (1+c_{\Omega'}+\lambda)^{\frac{m +n+\epsilon_0}{m_0}}\eta(\lambda),
 $$
 we define $\tr (A\eta(\cL)) :=\tr (A_1\eta_1(\cL_{\Omega'}))$.
\end{itemize}
In both cases, 
 this definition does not depend on $\Omega'$ or $\epsilon_0$,
 and coincides with $\tr (A\eta(\cL))$ when 
$\cL$ admits a unique self-adjoint 
extension,
or using Section \ref{sec_FC} when
 $\eta\in \cM^{m_\eta}$ with $m_\eta<-1$.
 Furthermore, we have
$\tr \left|A\ \eta(\cL)\right|
\leq C \|A\|_{\Psi^m, N}$ for some $N\in \bN$.
\end{lemma}

With this understanding for $\tr (A\eta(\cL))$, we can now start its analysis.
  
\begin{proposition}
\label{prop_term1}
We consider Setting \ref{set_cL_Rn}.
Let $A\in \Psi^{m}(\Omega)$ with integral kernel compactly supported in $\Omega\times \Omega$.
Let $\eta\in L^\infty(0,\infty)$ be compactly supported in $(0,\infty)$.
Then 
$$
|\tr \left(A\ \eta(t \cL)\right)|
\leq 
\left\{\begin{array}{l}
C_1 t^N\ \mbox{if} \ t\geq 1 \ \mbox{for some} \ N\in \bN,\\
C_2 t^{(m'-m)/m_0}\ \mbox{if} \ t\in (0,1), \ \mbox{for any} \ m'<-n,\\
\end{array}
 \right.
$$
and the constants $C_1,C_2$ may be chosen of the form $C_j=D_j  \sup_{(0,\infty)}|\eta|$ for some constant 
 $D_j$ depending on Setting \ref{set_cL_Rn}, $A$ 
 and  $\supp (\eta)$ and for $j=2$ on $m'$.

 If furthermore, $A\in \Psi^{m}_{cl}(\Omega)$ with $m=-n$, then 
$$
\lim_{t\to 0^+}
\tr \left(A\ \eta(t \cL)\right)
=\frac {\res (A)}{m_0}
\int_0^{+\infty}\!\!\!
\eta(u)    \frac{du}u .
$$
\end{proposition}

Routine arguments of localisation gives:
\begin{corollary}
\label{corM_prop_term1}
The same properties as in Proposition \ref{prop_term1}
hold in Setting \ref{set_cL_M} for any operator $A\in \Psi^m(M)$ and $\eta\in L^\infty(\bR)$ compactly supported in $(0,\infty)$.
\end{corollary}

\begin{proof}[Proof of Proposition \ref{prop_term1}]
We first prove the statement for $\eta\in C_c^\infty(0,\infty)$.
We continue with the notation of Lemma \ref{lem_trAetaOmega'}.
The properties of the trace and of pseudo-differential calculus then give the estimates:
\begin{align*}
&\left|\tr (A\ \eta(t \cL_{\Omega'}))\right|
\leq
\| A (\id+c_{\Omega'}+\cL_{\Omega'})^{-\frac {m}{m_0} }\|_{\sL(L^2(\Omega'))}
\tr |(\id+c_{\Omega'}+\cL_{\Omega'})^{\frac {m}{m_0} } \eta(t \cL_{\Omega'}) |
\\
&\qquad\lesssim 
\|(\id+c_{\Omega'}+\cL_{\Omega'})^{\frac {m}{m_0} } \eta(t \cL_{\Omega'}) \|_{\Psi^{m'}(\Omega'),N_1}
\lesssim 
\|\eta(t \cL_{\Omega'}) \|_{\Psi^{m'-m}(\Omega'),N_2}
\lesssim 
\|\eta(t \, \cdot)\|_{\cM^{(m'-m)/m_0},N_3},	
\end{align*}
for any $m'<-n$ by 
Theorem \ref{thm_f(L0)} (or rather Corollary \ref{cor_thm_f(L0)})
and Lemma \ref{lem_tr_Psim<-n}.
The $\cM$-semi-norm was defined in Definition \ref{def_cMm}, and we conclude with 
$$
\|\eta(t \, \cdot)\|_{\cM^{m_1},N_1}
\lesssim 
\left\{\begin{array}{l}
t^{N_1} \ \mbox{if} \ t\geq 1,\\
t^{m_1} \ \mbox{if} \ t\in (0,1).	
\end{array}\right.
$$
for any $m_1\in \bR$ and $N_1\in \bN$.
This  yields the estimate.

Let us now assume $A\in \Psi^{m}_{cl}$ with $m=-n$ and $\eta\in C_c^\infty(0,\infty)$. Then Theorem \ref{thm_f(L0)} and the  properties of pseudo-differential  operators  imply
$$
\lim_{t\to 0}
\tr \left(A\ \eta(t \cL)\right)
=
\lim_{t\to 0}
\tr \left(\Op\big(a_{m} \ \psi\ \eta(t \ell_{m_0}) \big) \right);
$$
here $a_m\in C^\infty(\Omega\times \bR^n\setminus\{0\})$ is homogeneous of degree $m=-n$ in $\xi$ and $x$-compactly supported in $\Omega'$,
and
 $\psi (\xi) =\psi_1(|\xi|)$ with $\psi_1\in C^\infty(\bR)$  
satisfying $0\leq \psi_1\leq 1$, 
$\psi(s)=0$ for $s\leq 1/2$ and $\psi(s)=1$ for $s\geq 1$.
By  Lemma \ref{lem_tr_Psim<-n}, we have
\begin{align*}
&\tr \left(\Op\left(a_{m}\ \psi \ \eta(t \ell_{m_0}) \right)\right)
=
\int_{\bR^n}\int_{\Omega'} a_{m}(x,\xi)\psi(\xi) \eta(t\ell_{m_0} (x,\xi))  dx d\xi
\\
&\qquad=
\int_{r=0}^{+\infty} \int_{(x,\xi)\in \Omega'\times \bS^{n-1}} 
a_{m}(x,r\xi)\ \psi_1 (r)\ \eta(t\ell_{m_0} (x,r\xi))  \ dx d\varsigma(\xi) \  r^{n-1}dr,
\end{align*}
after a change of variables in polar coordinates. 
We decompose the last integral as 
$\int_{r=0}^\infty = \int_{r=0}^1 + \int_{r=1}^{+\infty}$.
For the first integral, we have:
$$
|\int_{r=0}^1|\lesssim
\sup_{\substack{x\in \Omega' \\ |\xi|= 1}} |a_{m}(x,\xi)| 
\sup_{\substack{x\in \bar \Omega' \\ |\xi|= 1, 0<u\leq 1}}  |\eta(tu\ell_{m_0}(x,\xi))|,
$$
and the last supremum is zero
for $t$ positive but small enough.
For the second integral, using the homogeneity of the symbols
and $m=-n$,
we have
\begin{align*}
\int_{r=1}^{+\infty}
&=\int_{r=1}^{+\infty} \int_{(x,\xi)\in \Omega'\times \bS^{n-1}} 
a_{-n}(x,\xi) \eta(t r^{m_0} \ell_{m_0} (x,\xi))  \ dx d\varsigma(\xi) \  
 \frac{dr}r 
\\
&=
\frac 1{m_0}
\int_{(x,\xi)\in \Omega'\times \bS^{n-1}} 
 a_{-n}(x,\xi) 
  \int_{u=t \ell_{m_0}(x,\xi)}^{+\infty}
 \!\!\!
 \eta(u)  
\frac{du}u \ \ dx d\varsigma(\xi),
\end{align*}
after the change of variable $u=t r^{m_0} \ell_{m_0} (x,\xi)$.
For $t$ small enough, the second integral is in fact over $\int_{u=0}^\infty $, 
and the result follows.

If $\eta$ is not necessarily smooth but only in $L^\infty(0,\infty)$, 
then we construct a  sequence of smooth functions $\eta_k = \eta*\phi_{1/k}$ where $\phi\in C_c^\infty (\bR)$ is supported in $(-1,1)$ and $\int_\bR \phi (\lambda )d\lambda=1$.
The properties of the trace written as a sum over the eigenfunctions of $\cL_{\Omega'}$ implies easily the case of $\eta$ as $k\to +\infty$.

This  concludes the proof of Proposition \ref{prop_term1}.
 \end{proof}

\subsection{Proof of Theorem \ref{thm_main}}
\label{subsec_pf_thm_main}

By routing arguments of localisation, it suffices to prove the case of an open set $\Omega$ of $\bR^n$.
We continue with the notation of Lemma \ref{lem_trAetaOmega'}.
Since the part of the spectrum of $\cL_{\Omega'}$ involved in the expansion correspond to high frequencies, 
Theorem \ref{thm_f(L0)} implies that we may assume 
 $\cL_{\Omega'}\geq \id$ and then changing $\eta(\lambda)$ for $\eta(\lambda^{1/m_0})$ we may assume $m_0=1$; note that the proof can be carried out without these two assumptions but with cumbersome notation.
  By Proposition \ref{prop_term1} and its proof,
the operator  $A\ \eta(t \cL)$ is  trace-class for all $t\in \bR$;
furthermore, it suffices to show the case of an operator with symbol of the form  $a(x,\xi)= a_m (x,\xi) \psi(\xi)$
where $a_m\in C^\infty(\Omega\times (\bR^n\backslash\{0\}))$ is $m$-homogeneous in $\xi$, 
and compactly $x$-supported in $\Omega$  and the function $\psi\in C^\infty(\bR^n)$ 
is given by  $\psi(\xi)=\psi_1(|\xi|)$
with  $\psi_1(s)=1$ for $s\geq 1$ and $\psi_1(s)=0$ for $s\leq 1/2$.

By Theorem \ref{thm_f(L0)}, 
the symbol $b^{(t)}$ of $\eta(t\cL) $ 
admits an expansion
$b^{(t)}\sim \sum_j b^{(t)}_{-j}$
with 
$$
b^{(t)}_{-j}=
\sum_{k=0}^j c_{j,k}
\eta^{(j+k)}(t\ell_1),
$$	
and for any $N\in \bN$, $m'\in \bR$ and $M\in \bN_0$
$$
\|b^{(t)}-\sum_{j=0}^{N-1} b^{(t)}_{-j}\|_{S^{m'-N}_{loc},M}
\lesssim_{N,M} \| \eta(t\ \cdot )\|_{\cM^{m'}, M'}
$$
for some $M'$ depending on $N,M,m'$.
We estimate easily for any $m'\in \bR$,  $t\in (0,1)$ and $j\in \bN_0$
$$
\| \eta(t\ \cdot )\|_{\cM^{m'},M'} \lesssim_{\eta,m',M'}t^{m'}
\quad\mbox{and}\qquad
\|b^{(t)}_{-j}\|_{S^{m'}_{loc},M'}
 \lesssim_{j,\cL,\eta,m',M'}t^{j+m'}.
$$
 
Let $N\in \bN$ such that $\Re m-N<-n$.
We fix $m'<-n$ as close as we want to $-n$.
The properties of the pseudo-differential calculus and of the trace (see 
Lemma \ref{lem_tr_Psim<-n}) together with the estimates above
 imply for any $t >0$
$$
\tr( A \eta(t\cL) )
=\sum_{j+|\alpha|<N}
\frac {(2i\pi)^{-|\alpha|}}{\alpha !}
\tr \left(\Op(\partial_\xi^\alpha a \ \partial_x^\alpha b^{(t)}_{-j})\right)
+\cO(t^{m'-\Re m+N}).
$$
We compute easily
\begin{align*}
\tr \left(\Op(\partial_\xi^\alpha a \ \partial_x^\alpha b^{(t)}_{-j})\right)
&=
\int_{\Omega\times \bR^n}
\partial_\xi^\alpha a \ \partial_x^\alpha b^{(t)}_{-j}
\
dx d\xi
\\&=
\sum_{k=0}^j 
t^{j+k}
\int_{\Omega\times \bR^n}
\partial_\xi^\alpha a 
\ \partial_x^\alpha \left(
c_{j,k}
\eta^{(j+k)}(t\ell_1)\right) \
dx d\xi,
\end{align*}
We may write $\partial_x^\alpha \left(
c_{j,k}
\eta^{(j+k)}(t\ell_1)\right)$ as  a linear combination 
of 
$f_{k+p } \, t^p
\eta^{(j+k+p)}(t\ell_1)$
over $p=0,\ldots, |\alpha|$  
where each function 
$(x,\xi)\mapsto f_{k+p }(x,\xi)$ is in $C^\infty(\Omega\times (\bR^n\backslash\{0\}))$ 
and $(k+p)$-homogeneous in $\xi$;
in particular, $f_0$ is a constant.
Hence, $\tr \left(\Op(\partial_\xi^\alpha a \ \partial_x^\alpha b^{(t)}_{-j})\right)$ is a linear combination of 
 \begin{equation}
\label{eq_termjkpa}	
t^{j+k+p}	\int_{\Omega\times \bR^n}\!\!\!\!
\partial_\xi^\alpha a \ 
f_{k+p} \ 
\eta^{(j+k+p)}(t\ell_1) \ dx d\xi,
\qquad
0\leq k\leq j,\ 0\leq p \leq |\alpha|.
 \end{equation}
Proceeding  as in the proof of Proposition \ref{prop_term1}
and 
setting $m_{\alpha,j,k,p}:=m-|\alpha|+k+p $, 
\eqref{eq_termjkpa} is equal  for $t$ small enough to
$t^{j+k+p-(m_{\alpha,j,k,p}+n)}c_{m,|\alpha|,j,k,p}$
with 
$c_{m,|\alpha|,j,k,p} = c_{m,|\alpha|,j,k,p}(\eta) \times 
 c_{m,|\alpha|,j,k,p}(a)$
where
 \begin{align*}
 	c_{m,|\alpha|,j,k,p}(\eta)
 	&:=
 	\int_{u=0}^{+\infty}
 	\eta^{(j+k+p)}(u) \ u^{m_{\alpha,j,k,p}+n} \frac{du}u,
\\
c_{m,|\alpha|,j,k,p}(a)
&:=
\int_{\Omega \times \bS^{n-1}}
\partial^\alpha_\xi a_m \
f_{k+p} \ell_1^{-(m_{\alpha,j,k,p}+n)}
dx d\varsigma(\xi).
 \end{align*}
This shows that $\tr \left(\Op(\partial_\xi^\alpha a \ \partial_x^\alpha b^{(t)}_{-j})\right)$ is a multiple of 
$t^{j+|\alpha| -m  -n}$, 
and the expansion follows.

\medskip

The constant term in the expansion corresponds to the terms 
in \eqref{eq_termjkpa}	
with 
$j+k+p-(m_{\alpha,j,k,p}+n)=0$.
Integrations by parts show  in this case $c_{m,|\alpha|,j,k,p}(\eta)=0$
 vanishes unless $j+k+p=0$.
Hence, the constant term corresponds to the terms in \eqref{eq_termjkpa}	
 with $j=k=p=0=m_{\alpha,j,k,p}+n= m-|\alpha|+n$, 
but then in this case $c_{m,|\alpha|,j,k,p}(a)=0$ unless $\alpha=0$
by Lemma \ref{lem_intS=0}.
Therefore, the constant term corresponds to the terms in \eqref{eq_termjkpa}	 with 
$0=j=k=p=|\alpha|=m-n$. In other words, the constant term can only appear as the first term in the expansion of $\Op(a_{-n}\psi)$ which is given by  Proposition \ref{prop_term1} Part (3).

\smallskip

We observe that 
the constant $c_{m,|\alpha|,j,k,p}(\eta)$
is a multiple of $
\int_{u=0}^{+\infty}
\eta(u) \ u^{m-|\alpha|-j+n} \frac{du}u$ 
having integrated by parts repeatedly.
Hence, 
the constant $c_{m+n-j'}$ in the expansion is of the form 
$c_{m+n-j'}^{(a,\eta)}=
\tilde c_{m+n-j'}^{(\eta)}
\tilde c_{m+n-j'}^{(a)}$
where $\tilde c_{m+n-j'}^{(a)}$ is (universal) linear combinations over $|\alpha|+j=j'$
of the constants $c_{m,|\alpha|,j,k,p}(a)$ above and where
$$
\tilde c_{m+n-j'}^{(\eta)}:=
\int_{u=0}^{+\infty}
\eta(u) \ u^{m-j'+n} \frac{du}u .
$$

This concludes the proof of Theorem \ref{thm_main}.

\subsection{Proof  of Corollaries  \ref{cor_thm_main} and \ref{cor_thm_main_taylor0}}
\label{subsec_proof_cor}

By routing arguments of localisation, it suffices to prove the case of an open set $\Omega$ of $\bR^n$.

\begin{proof}[Proof of Corollary \ref{cor_thm_main} Part (1)]
Although the proof follows the same type of arguments given in Section \ref{subsec_pf_thm_main}, 
here, we  assume neither $m_0=1$ nor $\cL\geq \id$.
However,  it still suffices to consider 
 symbols $a$ of the form  $a(x,\xi)= a_m (x,\xi) \psi(\xi)$
where $a_m\in C^\infty(\Omega\times (\bR^n\backslash\{0\}))$ is $m$-homogeneous in $\xi$
and compactly $x$-supported, and the function $\psi\in C^\infty(\bR^n)$ 
is given by  $\psi(\xi)=\psi_1(|\xi|)$
with  $\psi_1(s)=1$ for $s\geq 1$ and $\psi_1(s)=0$ for $s\leq 1/2$.

By Theorem \ref{thm_f(L0)} and Corollary \ref{cor_thm_f(L0)} (see also Lemma \ref{lem_trAetaOmega'}), 
the symbol $b^{(t)}$ of $\eta(t\cL)$ is in $S^{m_\eta m_0}_{loc}$
and admits an expansion
$b^{(t)}\sim \sum_j b^{(t)}_{-j}$
with $b_0^{(t)}=\eta(t\ell_{m_0})$ and more generally
$$
b^{(t)}_{-j}=
\sum_{k=0}^j c_{j,k}
t^{j+k}\eta^{(j+k)}(t\ell_{m_0}),
\qquad j=0,1,2,\ldots
$$	
and for any $m'\in \bR$ and $M'\in \bN_0$
\begin{equation}
\label{eq_bt-sum}	
\|b^{(t)}-\sum_{j=0}^{N-1} b^{(t)}_{-j}\|_{S^{m_0m_\eta-N}_{loc},M'}
\lesssim_{N,M,\cL} \| \eta(t\ \cdot )\|_{\cM^{m_\eta}, M''},
\end{equation}
for some $M''$ depending on $N,M',m'$.
We estimate easily for any $M'\in \bN_0$ 
\begin{equation}
\label{eq_b-jt}
\|b_{-j}^{(t)}\|_{S^{m_\eta m_0 -j}_{loc},M'} \lesssim_{m',M',\cL,j} 
\|\eta(t\ \cdot)\|_{\cM^{m_\eta}, M''},
\qquad j=0,1,2,\ldots
\end{equation}
for some $M''\in \bN$, and
\begin{equation}
\label{eq_chit_cM}	
\|\eta(t\ \cdot)\|_{\cM^{m_\eta}, M'} \lesssim_{m_\eta,M'}t^{m_\eta}
\|\eta\|_{\cM^{m_\eta}, M'}.
\end{equation}

The properties of the pseudo-differential calculus and of the trace (see 
Lemma \ref{lem_tr_Psim<-n}) together with \eqref{eq_bt-sum}, \eqref{eq_b-jt}
and \eqref{eq_chit_cM}
 imply for any $t \in (0,1)$
$$
\tr( A \eta(t\cL) )
=
\sum_{j=0}^{N-1}
\tr \left(A  \Op( b_{-j}^{(t)})\right)
+\cO(t^{m_\eta})
$$
and for $j=0,1,\ldots,N-1,$
$$
\tr \left(A  \Op( b_{-j}^{(t)})\right)
=
\sum_{|\alpha|<N-j}
\frac {(2i\pi)^{-|\alpha|}}{\alpha !}
\tr \left(\Op(\partial_\xi^\alpha a \ \partial_x^\alpha b^{(t)}_{-j})\right)
+\cO(t^{m_\eta}).
$$
Proceeding as in Section \ref{subsec_pf_thm_main}, 
$\tr \left(\Op(\partial_\xi^\alpha a \ \partial_x^\alpha b^{(t)}_{-j})\right)$ is a (universal) linear combination of 
$$
I_t(\partial_\xi^\alpha a,\eta,j,k,p),
\quad \mbox{over}\ 0\leq k\leq j,\ 0\leq p \leq |\alpha|,
$$
where
\begin{align*}
&I_t(\tilde a ,\eta,j,k,p):=t^{j+k+p}
\int_{\Omega\times \bR^n}\!\!\!\!
\tilde  a \ 
f_{-j +(j+k+p)m_0} \ 
\eta^{(j+k+p)}(t\ell_{m_0}) \ dx d\xi
\\&\quad =t^{j+k+p}
\int_{\Omega\times (0,+\infty)\times \bS^{n-1}}
\!\!\!\!\!\!\!\!\!\!\!\!\!\!\!\!\!\!\!\!\!\!\!\!
(\tilde a \, f_{-j +(j+k+p)m_0}) \left(x,(\frac u{t\ell_{m_0}})^{\frac 1{m_0}} \xi\right)
\eta^{(j+k+p)}(u) 
\left(\frac u{t\ell_{m_0}}\right)^{\frac n {m_0}}
 dx  \frac{du}{m_0 u}  d\varsigma,
\end{align*}
having performed a change in polar coordinates $(r,\xi)\in (0,+\infty)\times \bS^{n-1}$ and another change of variables $u=t r^{m_0} \ell_{m_0}(x,\xi)$.
An easy calculation yields:
$$
I_t(\partial_\xi^\alpha a_m,\eta,j,k,p)=
t^{\frac {m-|\alpha| -j +n}{-m_0}}
\tilde c_{m-|\alpha|-j} ^{(\eta)} \tilde c_{|\alpha|,j,k,p} ^{(m,a)},
$$
where
$$
\tilde c_{|\alpha|,j,k,p} ^{(m,a)}
:=
\int_{\Omega\times  \bS^{n-1}}
\partial_\xi^\alpha a_m \, f_{-j +(j+k+p)m_0}
\ell_{m_0}^{\frac {m-|\alpha| -j -n}{-m_0}-(j+k+p)} dx d\varsigma.
$$
%
%
%
As long as $\frac{n+m-|\alpha| -j}{m_0}+j+k+p>0$,
we also obtain the estimates 
$$
|I_t(\partial_\xi^\alpha (a- a_m) ,\eta,j,k,p)|\leq C  
t^{j+k+p}
\quad \mbox{with}\quad C \lesssim_A 
\sup\{|\eta^{j+k+p}(\lambda)| \ : \ |\lambda| \leq  2^{m_0}\max_{\bar \Omega' \times \bS^{n-1}} \ell_0\}. 
$$

Combining the equalities and estimates above,
 $\tr \left(\Op(\partial_\xi^\alpha a \ \partial_x^\alpha b^{(t)}_{-j})\right)$ is a (universal) linear combination of 
$t^{\frac {m-|\alpha| -j +n}{-m_0}}
\tilde c_{m-|\alpha|-j} ^{(\eta)} \tilde c_{|\alpha|,j,k,p} ^{(m,a)},
$ over $0\leq k\leq j$, $0\leq p \leq |\alpha|$
modulo an error term $\cO(t^j)$.
Hence, 
$$
\tr( A \eta(t\cL) )
=
\sum_{j=0}^{N-1}
\sum_{|\alpha|<N-j}
t^{-\frac {m-|\alpha| -j -n}{m_0}} 
\tilde c_{m-|\alpha|-j} ^{(\eta)} 
{c'}_{m-|\alpha|-j} ^{(a)} 
+\cO(t^{m_\eta}),
$$
for some constants ${c'}_{m-|\alpha|-j} ^{(a)}$.
However,
the uniqueness of the asymptotic and the case of $\eta$ compactly supported in $(0,+\infty)$ in Theorem \ref{thm_main} yield
${c'}_{m-|\alpha|-j} ^{(a)} = \tilde c_{m-|\alpha|-j} ^{(a)} $.
\end{proof}

\begin{proof}[Proof of Corollary \ref{cor_thm_main} Part (2)]
We consider a dyadic decomposition of $(0,+\infty)$, 
that is, $\theta_0\in C_c^\infty(\bR)$ supported in $[1/2,2]$ such that 
$$
\forall \lambda>0\qquad \sum_{k\in \bZ}\theta_k(\lambda) =1, 
\quad\mbox{where}\quad
\theta_k(\lambda):= \theta_0(2^{-k}\lambda).
$$ 
Let $\eta\in \cM^{m_\eta}(\bR)$.
We have 
$$
\forall \lambda>0\qquad 
\eta(\lambda) = \sum_{k\in \bZ} \tilde \eta_k (\lambda),
\quad\mbox{where}\quad
\tilde \eta_k(\lambda):= \eta_k(2^{-k}\lambda), \quad
\eta_k (\lambda) :=\eta(2^{k} \lambda)\theta_0(\lambda).
$$
Note that for every $N_1\in \bN_0$
$$
\forall k\in \bN_0\qquad 
\sup_{j=0,\ldots,N_1} \|\eta_k^{(j)}\|_{L^\infty}
\leq C_{N_1} 2^{km_\eta}
\|\eta\|_{\cM^{m_\eta},N_1},
$$
with a constant $C_{N_1}$ independent of $k\in \bN_0$ or $\eta\in \cM^{m_\eta}$.
Since $\supp (\eta_k)\subset \supp(\theta_0)\subset [1/2,2]$ for any $k\in \bZ$, 
the application of Theorem \ref{thm_main} to each $\eta_k \in C_c^\infty(\bR)$ gives
for any $t\in (0,1)$
$$
|\tr \left(A \eta_k(t\cL)\right)
-\sum_{j=0}^{N-1} c_{m+n-j}(A,\eta_k) t^{-\frac{m+n-j}{m_0}}|
\leq C' \sup_{0\leq j \leq N_1} \|\eta_k^{(j)}\|_{L^\infty}  t^{\frac{-m-n+N}{m_0}}, 
$$
where the constant $C$ and the integer $N_1\in \bN_0$ depend $N$ and the setting.
Furthermore, 
$$
c_{m+n-j}(A,\eta_k)=
\tilde c_{m+n-j}^{(\eta_k)}
\tilde c_{m+n-j}^{(A)}
= 
2^{-k \frac{m-j+n}{m_0} } \tilde c_{m+n-j}^{(\tilde \eta_k)}\tilde c_{m+n-j}^{(A)}
=
2^{-k \frac{m-j+n}{m_0} }  c_{m+n-j}(A,\tilde\eta_k),
$$
so for any $j\in \bN_0$
$$
\sum_{k=0}^{+\infty}
2^{k \frac{m-j+n}{m_0} } c_{m+n-j}(A,\eta_k)
=
c_{m+n-j}(A,\eta),
$$
when  $\supp(\eta)\subset [1,+\infty)$
and $\int_{u=0}^{+\infty}
|\eta(u)| \ u^{\frac{m+n}{m_0}} \frac{du}u<\infty$.
Under these hypotheses, 
we may  apply the  estimate above to $2^{-k}t$ for every $k\in \bN_0$. Summing up over $k\in \bN_0$, we obtain for all $t\in (0,1)$
$$
|\tr \left(A \eta(t\cL)\right)
-\sum_{j=0}^{N-1} c_{m+n-j}(A,\eta) t^{-\frac{m+n-j}{m_0}}|
\leq C' C_{N_1} \|\eta_k\|_{\cM^{m_\eta},N_1}  
\sum_{k\in \bN_0} 2^{km_\eta} (2^{-k} t)^{\frac{-m-n+N}{m_0}}.
$$
The result follows when $m_\eta  +(m+n-N)/m_0<0$, 
and therefore for all $N\in \bN$.
\end{proof}

This concludes the proof of Corollary \ref{cor_thm_main}.
We observe that its argument can be pushed further to obtain the property in  Corollary \ref{cor_thm_main_taylor0}.

\begin{proof}[Proof of Corollary \ref{cor_thm_main_taylor0}]
By Theorem \ref{thm_main} and Lemma \ref{lem_cLOmega'}, we may assume that $\supp (\eta)\subset [c,1]$ for the constant $c=\min(0,c_{\Omega'})$.
For such a function $\eta$, we have
$$
\|\eta (t\, \cdot)\|_{\cM^{m'}	,N}
\leq C \max_{j=0,\ldots,N_1}\| \eta^{(j)}\|_{L^\infty} t^{m'}, 
$$
for some constant $C>0$ and integer $N_1$ depending on $m'$ and $N$ but not on $\eta$.
Following the arguments of Theorem \ref{thm_main} and Corollary \ref{cor_thm_main} gives the desired property.
\end{proof}

\subsection{Proof of Theorem \ref{thm_TR}}
\label{subsec_proof_thm_TR}
This section is devoted to the proof of Theorem \ref{thm_TR}. 
Routine arguments of localisation imply that it suffices to prove the case of an open set $\Omega$.
Let $A\in \Psi^{m}_{cl}(\Omega)$ with $m\in \bC$,
$m\not \in \bZ$ and $m_1:=\Re m>  -n$, 
and with a symbol compactly $x$-supported in an open set $\Omega'$ which is relatively compact in $\Omega$.
Let $N\in \bN$ denote the largest integer such that
$N> \Re m+n$.
Let $\eta\in C_c^\infty(\bR)$.
We understand $\tr (A\eta(\cL))$ as in  Lemma \ref{lem_trAetaOmega'}
and 
we may assume $\supp(\eta)\subset [-c,1]$
because of Theorem \ref{thm_main} with $c=\min (0,c_{\Omega'})$.
It suffices to consider 
 symbols $a$ of the form  $a(x,\xi)= a_m (x,\xi) \psi(\xi)$
where $a_m\in C^\infty(\Omega\times (\bR^n\backslash\{0\}))$ is $m$-homogeneous in $\xi$
and compactly $x$-supported in  $\Omega'$, and the function $\psi\in C^\infty(\bR^n)$ 
is given by  $\psi(\xi)=\psi_1(|\xi|)$
with  $\psi_1(s)=1$ for $s\geq 1$ and $\psi_1(s)=0$ for $s\leq 1/2$.

\subsubsection{Case of $\eta\equiv 1$ near 0.}
In this case, we modify the arguments given in the proof of Corollary \ref{cor_thm_main}.
Here, the semi-norms in $\eta$ would not provide any decay in $t$. However, since $\eta \equiv 1 $ near 0, we have for all $m_1\geq 0$ and $N\in \bN$:
$$
\forall t\in (0,1)\qquad
\|(\eta-1)(t \, \cdot)\|_{\cM^{m_1},N}
\lesssim_{N,m_1,\eta} 
t^{m_1} 
$$

By  Theorem \ref{thm_f(L0)} applied to $\eta-1$,
the symbol $b^{(t)}$ of $\eta(t\cL)$ is smoothing
and admits an expansion
$b^{(t)}\sim \sum_j b^{(t)}_{-j}$
with $b_0^{(t)}=\eta(t\ell_{m_0})$ and more generally
$$
b^{(t)}_{-j}=
\sum_{k=0}^j c_{j,k}
t^{j+k}\eta^{(j+k)}(t\ell_{m_0}),
\qquad j=0,1,2,\ldots.
$$	
 Applying  Theorem \ref{thm_f(L0)} to $\eta-1$ implies that,  for any $m'\geq 0$ and $N,M'\in \bN$, we have for all $t\in (0,1)$:
$$
\|b^{(t)}-\sum_{j=0}^{N-1} b^{(t)}_{-j}\|_{S^{m'-N}_{loc},M'}
\lesssim_{N,M,\cL} \| (\eta-1)(t\ \cdot )\|_{\cM^{m' m_0}, M''}
\lesssim
 t^{m'm_0}.
 $$
 We fix $m'>0$ small enough so that 
$N>\Re m +m_0 m' +n$.
A modification of  the  proof of Corollary \ref{cor_thm_main} Part (1) gives that
$$
\tr( A \eta(t\cL) )
=
\sum_{j+|\alpha|<N}
\frac {(2i\pi)^{-|\alpha|}}{\alpha !}
\tr \left(\Op(\partial_\xi^\alpha a \ \partial_x^\alpha b^{(t)}_{-j})\right)
+\cO(t^{m_0 m'}),
$$
and  $\tr \left(\Op(\partial_\xi^\alpha a \ \partial_x^\alpha b^{(t)}_{-j})\right)$ is a linear combination of 
$t^{\frac {m-|\alpha| -j +n}{-m_0}}$ over $0\leq k\leq j$, $0\leq p \leq |\alpha|$
modulo an error term $\cO(t^j)$.
Let us study more closely the case of $j=0$.
For $\alpha\not=0$, we have:
$$
\tr \left(\Op(\partial_\xi^\alpha a \ \partial_x^\alpha b^{(t)}_{0})\right)
=
\int_{\Omega\times \bR^{n}}
\partial_\xi^\alpha a \ \partial_x^\alpha \eta(t \ell_{m_0})\
dx d\xi
=\sum_{p=1}^{|\alpha|}t^p I_p(a)
$$
where
\begin{align*}
	I_p(a) &:= 
	\int_{\Omega\times \bR^{n}}
(\partial_\xi^\alpha a) \,  f_{pm_0} \,  \eta^{(p)}(t \ell_{m_0})\
dx d\xi
\\
&=\int_{\Omega\times \bS^{n-1}}
\int_{u=0}^{+\infty}
\left((\partial_\xi^\alpha a) \,  f_{pm_0} \right)\left( x,\left(\frac{u}{t\ell_{m_0}}\right)^{\frac 1{m_0}} \xi\right) \,  \eta^{(p)}(u) \left(\frac{u}{t\ell_{m_0}}\right)^{\frac n{m_0}} 
\frac{du}{m_0 u}\
dx d\varsigma(\xi).
\end{align*}
We compute easily that $I_p(a_m)$ is equal to $t^{-\frac{m-|\alpha|+pm_0+n}{m_0}}$ up to a finite constant and that
$$
	|I_p(a-a_m) |
\lesssim 
\int_{\Omega'\times \bS^{n-1}}
\int_{u=0}^{2^{m_0}t\ell_{m_0}}
\left(\frac{u}{t\ell_{m_0}}\right)^{\frac {m_1 -|\alpha|+pm_0 +n}{m_0}} 
|\eta^{(p)}(u)| \frac{du}{u} \ dx d\varsigma(\xi)
$$
is identically 0 for $t\in (0,t_0)$ with $t_0$ small enough.
For $\alpha=0$, we see
$$
\tr \left(\Op(a \  b^{(t)}_{0})\right)
=
\int_{\Omega\times \bR^{n}}
 a \ \eta(t \ell_{m_0})\
dx d\xi,
$$
and that $\int_{\Omega\times \bR^{n}}
 a_m \ \eta(t \ell_{m_0})\
dx d\xi $ is equal to $t^{-\frac{m+n}{m_0}}$ up to a finite constant.
We are therefore led to analyse the error term
$$
E_t:=\int_{\Omega\times \bR^n}
(a -a_m) \ \eta(t\ell_{m_0}) \
dx d\xi .
$$
Recall that $a=a_m\psi$.
For each $x\in \Omega$ and $t\in (0,1)$, 
the Fourier inversion formula for tempered distribution implies
\begin{equation}
\label{eq_pf_thm_TROmega_inty_chi}
\int_{\bR^n}
(a_m \psi -a_m)(x,\xi)\ \eta(t\ell_{m_0})(x,\xi)\
d\xi
=
 \int_{\bR^n}(\kappa_{a,x} - \kappa_{a_m,x})(y) \
 f_{t,x}(y) \ dy \ 
\end{equation}
having used the notation of Section \ref{subsubsec_def_TR}
and set
$$
f_{t,x}(y):= \cF^{-1} \left\{\eta(t\ell_{m_0}(x,\cdot))\right\}(-y).
$$
The last integral in \eqref{eq_pf_thm_TROmega_inty_chi} tends to 
$(\kappa_{A,x} - \kappa_{a_m,x})(0)=\TR_x(A)$ as $t\to 0$
since   the function $\kappa_{a,x} - \kappa_{a_m,x}$  is continuous and bounded on $\bR^n$ and $(f_{t,x})$ is a Schwartz approximation of the identity in the sense that  
$$
f_{t,x} (y)= t^{-\frac n {m_0}} f_{1,x}(t^{-\frac 1{m_0}} y),
\quad
f_{1,x}\in \cS(\bR^n),
\quad
\int_{\bR^n} f_{1,x}(y) dy = \eta(\ell_{m_0}(x,0)) =\eta(0)=1.
$$
One checks easily that each $\cS(\bR^n)$-semi-norm of $f_{1,x}$ 
and the supremum norm of $\kappa_{a,x} - \kappa_{a_m,x}$ 
are uniformly bounded with respect to $x\in \bar \Omega'$. 
Therefore, the convergence of \eqref{eq_pf_thm_TROmega_inty_chi} to 0 is uniform with respect to $x$ as $t\to 0$ and we have 
$$
E_t =  \int_{\Omega'} \int_{\bR^n}(\kappa_{a,x} - \kappa_{a_m,x})(y) \ f_{t,x}(y) \ dy \ dx 
 \longrightarrow_{t\to 0} \int_{\Omega'} \TR_x(A) \ dx = \TR(A).
 $$
Theorem \ref{thm_TR} in the case of $\eta\equiv 1$ near 0 follows.

\subsubsection{Case of $m+km_0\not\in \bZ$ for all $k\in \bN$}
We fix  $\eta_0\in C_c^\infty(\bR)$ valued in $[0,1]$ and such that $\eta_0\equiv 1$ on $\supp (\eta)$.
We denote by $N'$ the integer of Corollary \ref{cor_thm_main_taylor0} Part (2) for $N$.
We set $\eta_1:=\eta - p\eta_0$ where $p$ is the Taylor expansion of $\eta$ of order $N$, that is,
$$
p(\lambda)=\sum_{k=0}^{N'} \frac{ \eta^{k}(0)}{k!} \lambda^k.
$$
Hence we can write 
$$
\eta = \eta_1 + p \eta_0,
$$
with $\eta_1\in C_c^\infty(\bR)$ satisfying the hypotheses of Corollary \ref{cor_thm_main_taylor0} Part (2) for $N$.
We see
$$
\tr (A \eta(t\cL)) 
=
\tr (A \eta_1(t\cL))
+
\sum_{k=0}^{N'} \frac{ \eta^{k}(0)}{k!}t^k 
\tr (A\cL^k \eta_0(t\cL)).
$$
The hypotheses on $m$ allows us to apply the case already proved above for every term in the sum over $k$.
Theorem \ref{thm_TR} in this case follows by linearity of $\TR$.

\appendix

\section{Proof of Theorem \ref{thm_f(L0)}}
\label{sec_pf_thm_f(L0)}

This section is devoted to the proof of Theorem \ref{thm_f(L0)}.
We consider  Setting \ref{set_cL_Rn}.
We start with the construction of a parametrix for $\cL-z$, 
thereby obtaining a pseudo-differential expansion for the resolvent of $\cL$ with a precise behaviour in $z$.

\subsection{Construction of a parametrix $P_z$}
\label{subsec_parametrix}

We fix a cut-off function $\psi$ for the low frequencies, that is, 
$\psi\in C^\infty(\bR^n)$ with $\psi(\xi)=0$ for $|\xi|\leq 1/2$ and $\psi(\xi)=1$ for $|\xi|\geq 1$.

\begin{lemma}
\label{lem_est_l-z-1}
The symbol $(\ell_{m_0}-z)^{-1}\psi$ given by 
$(x,\xi)\mapsto (\ell_{m_0}(x,\xi)-z)^{-1}\psi(\xi)$
is in $S^{-m_0}_{loc}(\Omega\times \bR^n)$.
Furthermore, 
for any $N\in \bN$,
there exists $C\geq 0$ such that 
$$
\forall z\in \bC \setminus \bR
\ \mbox{with}\ |\Im z|\leq 2 \qquad
\|(\ell_{m_0}-z)^{-1}\psi\|_{S^{-m_0}_{loc},N}
\leq C \left(\frac{\langle z\rangle }{|\Im z|}\right)^{N+1}.
$$
\end{lemma}

\begin{proof}
Observing that 
$$
\forall z\in \bC\backslash \bR,
\quad\forall (x,\xi)\in K\times (\bR^n\backslash\{0\})
\qquad
|(\ell_{m_0}(x,\xi)-z)^{-1}| \leq \frac1{c_K |\xi|^{m_0}+1}
 \frac {\langle z\rangle }{|\Im z|},
$$
the statement follows from routine computations.
\end{proof}

We use the notation and result of Proposition \ref{prop_OpPhi}.
For each $z\in \bC \setminus\bR$, we set 
$$
R_{z}:= (\cL -z) \ \Op_\Phi\left( (\ell_{m_0}-z)^{-1} \psi \right) \ -  \ \id.
$$
This defines an operator $R_{z}\in \Psi^0_{ps}(\Omega)$. Let us show that it is in fact of order -1 and estimate the $z$-dependence of its semi-norms.
We write 
$$
R_{z} = \tilde R_{z} + \check R_{z} - \id + \Op_\Phi(\psi^2),
$$
where 
\begin{align*}
 \tilde R_{z}
 &:=\Op_\Phi\left( (\ell_{m_0}-z) \psi \right) \Op_\Phi\left( (\ell_{m_0}-z)^{-1} \psi \right) - \Op_\Phi(\psi^2),\\
 \check R_{z}
 &:= 
 \left(\cL - \Op_\Phi\left(  \ell_{m_0} \psi \right)\right) 
 \Op_\Phi\left( (\ell_{m_0}-z)^{-1} \psi \right).
\end{align*}
Clearly, $\id - \Op_\Phi(\psi^2)\in \Psi_{ps}^{-\infty}(\Omega)$.
Let us analyse $\check R_{z}$.
The operator $\cL - \Op_\Phi\left(  \ell_{m_0} \psi \right)$ is  in $\Psi^{m_0}_{ps}(\Omega)$
as the difference of two operators in $\Psi^{m_0}_{ps}(\Omega)$. However, modulo $S_{loc}^{-\infty}(\Omega\times\bR^n)$, its symbol is $\ell - \ell_{m_0}\psi$ which is in $S^{m_0-1}_{loc}(\Omega\times\bR^n)$.
Hence $\cL - \Op_\Phi\left(  \ell_{m_0} \psi \right)\in \Psi^{m_0-1}_{ps}(\Omega)$.
By Lemma \ref{lem_est_l-z-1}, $\Op_\Phi\left( (\ell_{m_0}-z)^{-1} \psi \right) \in \Psi^{-m_0}_{ps}(\Omega)$
so $S^{-1,z} \in \Psi^{-1}_{ps}(\Omega)$ with for any $N\in \bN$
$$
 \| \check R_{z}\|_{\Psi^{-1},N}
 \lesssim 
 \|\cL - \Op_\Phi\left(  \ell_{m_0} \psi \right)\|_{\Psi^{m_0-1},N_1}
 \|(\ell_{m_0}-z)^{-1} \psi \|_{S^{-m_0}_{loc},N_2}
 \lesssim \left(\frac{\langle z\rangle }{|\Im z|}\right)^{N_2+1},
$$
for some $N_1,N_2$ independent of $z\in \bC\setminus \bR$.
For $\tilde R_{z}$, by Proposition \ref{prop_OpPhi} and Lemma \ref{lem_est_l-z-1}, 
$$
 \| \tilde R_{z}\|_{\Psi^{-1},N}
 \lesssim 
 \|\Op_\Phi\left( (\ell_{m_0}-z) \psi \right) \|_{\Psi^{m_0},N'_1}
\|\Op_\Phi\left( (\ell_{m_0}-z)^{-1} \psi \right) 
 \|_{\Psi^{-m_0},N'_2}
 \lesssim \langle z\rangle \left(\frac{\langle z\rangle }{|\Im z|}\right)^{N'_2+1}.
$$
We have obtained that for any $N\in \bN$, there exist  $C=C_{(N)}\geq 0$ and $N'=N'_{(N)}\in \bN$ depending on $N$ (as well as Setting \ref{set_cL_Rn}, $\psi$ and $\Phi$) but not on $z\in \bC\setminus\bR$ such that 
$$
\|R_{z}\|_{\Psi^{-1},N}\leq C \Big(1+\frac{\langle z\rangle^{N'+1} }{|\Im z|^{N'}}\Big).
$$

For $t>0$ we set $\psi_t(\xi)=\psi(t\xi)$.
We modify the argument of Lemma \ref{lem_expansion}.
For any $j,k,N\in \bN_0$, we have with some constants $D_{j,k,N}$
$$
\forall t\in (0,1)\qquad
\forall a\in S^{-j}_{loc}(\Omega\times \bR^n)\qquad
\|a \psi_t\|_{S^{-j+k}_{loc},N}\leq D_{j,k,N} t^k\|a\|_{S^{-j}_{loc},N},
$$
while the property of the pseudo-differential calculus implies that for any $j,N\in \bN$ there exist $C_{j,N}>0$  and $N_{j,N}\in \bN$ such that 
$$
\forall A\in \Psi^{-1}_{ps}(\Omega)
\qquad
\|A^j\|_{\Psi^{-j},N} \leq C_{j,N} (\|A\|_{\Psi^{-1},N_{j,N}} )^j;
$$
If $A=\Op(a)$ and $j\in \bN$ then we denote by $a^{\# j}$  the symbol of $A^j = \Op(a^{\# j})$ with the convention that $a^{\# 0}=1$ and $A^0=\id$.
We denote by $r_{z}$ the symbol of $R_{z}=\Op(r_{z})$.
The previous estimates yield
\begin{equation}
\label{eq_est_rzj}
\|r_{z}^{\# j} \psi_t\|_{S^{-j+k}_{loc},N}
\leq D_{j,k,N} C_{j,N}  
\left(C_{(N_{j,N})} \Big(1+\frac{\langle z\rangle^{N_{(N_{j,N})}'+1} }{|\Im z|^{N_{(N_{j,N})}'}}\Big)
\right)^j \ t^k.	
\end{equation}
Choosing a decreasing sequence $(t_j)_{j\in \bN}\subset (0,1)$ 
such that 
$$
D_{j,k,N} C_{j,N}  
\left(C_{(N_j)} \Big(1+j^{2N_{(N_{j,N})}'+1}\Big)\right)^j \ t_j^k
\leq 2^{-j}
\quad\mbox{with} \quad j=k=N, 
$$
the series $\sum_{j=0}^{+\infty}r_{z}^{\# j} \psi_{t_j}$ is absolutely convergent in $S^0_{loc}(\Omega\times \bR^n)$ and for all $j_0\in \bN$ we have 
$$
\langle z\rangle \leq j_0 \ \mbox{and}\ 
|\Im z| > 1/j_0\ \Longrightarrow \ 
 \sum_{j>j_0}\|r_{z}^{\# j} \psi_{t_j}\|_{S^{-j_0}_{loc},j_0} \leq 2^{-j_0}.
 $$
We now define 
$$
P_z := 
\Op_\Phi\left( (\ell_{m_0}-z)^{-1} \psi \right) 
\Op_\Phi\left( \sum_{j=0}^{+\infty}r_{z}^{\# j} \psi_{t_j} \right) .
$$
By construction,
for each $z\in \bC\setminus\bR$, 
$$
P_z\in \Psi_{ps}^{-m_0}
\quad\mbox{and}\quad
(\cL -z) P_z -\id \in \Psi^{-\infty}_{ps}(\Omega).
$$
Moreover,  the mapping $z\mapsto P_z$ is continuous, even holomorphic, from $\bC \setminus \bR$ to  $\Psi_{ps}^{-m_0}$.

The following technical lemma summarises the properties of $P_z$ we will need:
 \begin{lemma}
\label{lem_Pz}
\begin{enumerate}
\item 
For each $z\in \bC\setminus\bR$,
the symbol $p_z\in S^{-m_0}_{loc}(\Omega\times \bR^n)$ of $P_z$ admits an expansion 
$$
p_z \sim \sum_{j\geq 0} p_{z,-m_0-j}
$$
where each $p_{z,-m_0-j}\in S^{-m_0-j}_{loc}(\Omega\times \bR^n)$ is of the form 
$$
p_{z,-m_0-j} = 
\sum_{k=0}^{2j} \frac{d_{j,k}}{(\ell_{m_0}-z)^{1+k}},
$$
with each $d_{j,k}\in S^{-j+k m_0}_{loc} (\Omega\times \bR^n)$ being a symbol independent of $z$, 
identically 0 for $|\xi| <1$
and
homogeneous of degree $-j+ km_0$ in $\xi$ for  $|\xi|\geq 2$; 
when $j=0$, $d_{0,0}$ is identically equal to 1 for $|\xi|\geq 2$. 

\item Let $f\in \cM^m(\bR)$ with $m<-1$.
We construct an almost analytic extension $\tilde f$ as in Lemma \ref{lem_tildefcMm}. 
We can modify the sequence $(t_j)$ chosen above in such a way that for every $N\in \bN$,
 the integrals
\begin{align*}
I^{(1)}_N
&:=\int_{\bC}
\frac { |\bar \partial \tilde f (z)| }
 {|\Im z|} \ 
\|P_z\|_{\Psi^{-m_0},N}
\ L(dz), 
\\
I^{(2)}_N
&:=	
\int_{\bC}
\frac { |\bar \partial \tilde f (z)| }
 {|\Im z|} \ 
 \|P_z - \sum_{j<N}\Op_\Phi (p_{z,-m_0-j})\|_{\Psi^{-m_0-N},N}
\ L(dz),
\\
I^{(3)}_N
&:=\int_{\bC}
\frac { |\bar \partial \tilde f (z)| }
 {|\Im z|} \ 
\|\id  - (\cL-z) P_z\|_{\Psi^{-N},N}
\ L(dz),
\end{align*}
are finite.
\end{enumerate}
\end{lemma}

\begin{proof}
By Proposition \ref{prop_OpPhi},
the symbol $r_{z}$ admits the expansion
$r_{z}\sim \sum_{j> 0} r_{z,-j}$
where
$$
r_{z,-j} =
\sum_{|\alpha|+j'  = j }
	\frac {(2i\pi)^{-|\alpha|}}{\alpha !}
	\partial^\alpha_\xi  \ell_{m_0-j'}
\partial^\alpha_x	 (\ell_{m_0}-z)^{-1} \psi ,
$$
so the symbol $b_z$ of $B_z=\Op_\Phi\left( \sum_{j=0}^{+\infty}r_{z}^{\# j} \psi_{t_j} \right)$
admits an expansion of the form $b\sim\sum_{j\geq 0} b_{z,-j}$ with $b_{z,0}=1$ and for any $j\in \bN$
$$
b_{z,-j} = 
\sum_{k=0}^j \frac{\tilde d_{j,k}}{(\ell_{m_0}-z)^{1+k}},
$$
where each $\tilde d_{j,k}\in S^{-j+(1+k)m_0}_{loc} (\Omega\times \bR^n)$ is homogeneous of degree $-j+(1+k)m_0$ in $\xi$ for large $|\xi|$ and identically 0 for small $|\xi|$.
Hence, 
	the symbol $p_z$ admits the expansion $\sum_{j\geq 0} p_{z,-m_0-j}$ with 
	$$
	p_{z,-m_0-j} = \sum_{|\alpha| +j' = j }
	\frac {(2i\pi)^{-|\alpha|}}{\alpha !}
	\partial^\alpha_\xi (\ell_{m_0}-z)^{-1}  
\partial^\alpha_x b_{z,-j'},
	$$
and it is therefore as described in the statement.
 Part (1) is proved.

Let us prove Part (2).
We start with the following consequence of
Lemmata \ref{lem_est_l-z-1} and \ref{lem_tildefcMm} and estimates  \eqref{eq_est_rzj}  
$$
\int_{\bC}
\frac { |\bar \partial \tilde f (z)| }
 {|\Im z|} \ 
 \frac{\langle z\rangle^{M_1} }{|\Im z|^{M_1}} \
 \|r_{z}^{\# j} \psi_t\|_{S^{-j+k}_{loc},M_2}
\ L(dz), 
\leq C'_{j,k,M_1,M_2}  
 \|f\|_{\cM^m, 2jN_{(N_{j,M_2})}'+j+2M_1+5}
  t^k,
$$
 for any $j,k,M_1,M_2$ and
for some constants $C'_{j,k,M_1,M_2} $.
If need be, we modify the sequence $(t_j)$ so that it also satisfies 
$$
C'_{j,k,M_1,M_2}  
 \|f\|_{\cM^m, 2jN_{(N_{j,M_2})}'+j+2M_1+5}
  t_j^k \leq 2^{-j}
  \quad\mbox{with} \quad j=k=M_1=M_2.
$$
In this way, for every  $M_1,M_2\in \bN_0$, the sum 
$$
\sum_{j\geq 0}
\int_{\bC}
\frac { |\bar \partial \tilde f (z)| }
 {|\Im z|} \ 
 \frac{\langle z\rangle^{M_1} }{|\Im z|^{M_1}} \
 \|r_{z}^{\# j} \psi_{t_j}\|_{S^{0}_{loc},M_2}
\ L(dz)
$$
is finite.

Let us consider the integrals $I^{(1)}_M$ for $M\in \bN$.
We have 
$$
\|P_z\|_{\Psi^{-m_0},M} 
\lesssim 
\|(\ell_{m_0}-z)^{-1} \psi \|_{\Psi^{-m_0},M_1}
\sum_{j\geq 0}\| r_{z}^{\# j} \psi_{t_j}\|_{\Psi^{0},M_2},
$$
for some $M_1,M_2$.
By Lemma \ref{lem_est_l-z-1}, 
$$
\|(\ell_{m_0}-z)^{-1} \psi\|_{S_{loc}^{-m_0},M_1}
\lesssim 
\frac{\langle z\rangle^{M'_1} }{|\Im z|^{M'_1}},
$$
so 
$$
I^{(1)}_N
\lesssim 
\sum_{j\geq 0}
\frac { |\bar \partial \tilde f (z)| }
 {|\Im z|} \ 
 \frac{\langle z\rangle^{M'_1} }{|\Im z|^{M'_1}}
 \| r_{z}^{\# j} \psi_{t_j}\|_{\Psi^{0},M_2}
\ L(dz), 
$$
 is finite.

We generalise this when considering $I^{(2)}_M$, $M\in \bN$. Indeed, we have
\begin{align*}
&\|P_z - \sum_{|\alpha|<M}\frac {(2i\pi)^{-|\alpha|}}{\alpha !}
\Op_\Phi\left(\partial^\alpha_\xi (\ell_{m_0}-z)^{-1} \psi
\partial^\alpha_x
(\sum_{j=0}^{+\infty}r_{z}^{\# j} \psi_{t_j})\right)
\|_{\Psi^{-m_0-M},M} 
\\&\qquad\lesssim 
\|(\ell_{m_0}-z)^{-1} \psi \|_{\Psi^{-m_0},M_1}
\sum_{j\geq 0}\| r_{z}^{\# j} \psi_{t_j}\|_{\Psi^{0},M_2},
\end{align*}
Proceeding as above, the right hand side is integrable 
against ${ |\bar \partial \tilde f (z)| }/
 {|\Im z|} $.
Note that 
 $$
 \|\Op_\Phi\left(\partial^\alpha_\xi (\ell_{m_0}-z)^{-1} \psi
\partial^\alpha_x
r_{z}^{\# j} \psi_{t_j})\right)\|_{\Psi^{-m_0-M},M}
\lesssim 
\|(\ell_{m_0}-z)^{-1} \psi \|_{\Psi^{-m_0},M'_1}
\|\partial^\alpha_x
r_{z}^{\# j} \psi_{t_j}\|_{\Psi^{-M},M'_2}, 
$$
for some $M'_1,M'_2$, and this is integrable 
 against ${ |\bar \partial \tilde f (z)| }/
 {|\Im z|} $.
Classical analysis shows that 
 $$
\|\sum_{|\alpha|<M}\frac {(2i\pi)^{-|\alpha|}}{\alpha !}
\Op_\Phi\left(\partial^\alpha_\xi (\ell_{m_0}-z)^{-1} \psi
\partial^\alpha_x
(\sum_{j=0}^{M'_2}r_{z}^{\# j} \psi_{t_j})\right) -
\sum_{j<M}\Op_\Phi (p_{z,-m_0-j})
\|_{\Psi^{-m_0-M},M} 
$$
is then bounded up to a constant by a finite sum of $\langle z\rangle^{M''_1}/ |\Im z|^{M''_2}$, 
which is integrable 
 against ${ |\bar \partial \tilde f (z)| }/
 {|\Im z|} $.
 Hence $I^{(2)}_M$ is finite.

Let us consider the integral $I^{(3)}_N$, $N\in \bN$.
We decompose
$\|\id  - (\cL-z) P_z\|_{\Psi^{-N},N}
\leq E_{1,z} +E_{2,z}$ with 
\begin{align*}
E_{1,z}&:=
\|\id  - (\cL-z) \Op_\Phi\left( (\ell_{m_0}-z)^{-1} \psi \right) \|_{\Psi^{-N},N}
\\
E_{2,z}&:=
\|\Op_\Phi\left( (\ell_{m_0}-z)^{-1} \psi \right) 
\Op_\Phi\Big( \sum_{j>0
}r_{z}^{\# j} \psi_{t_j} \Big) \|_{\Psi^{-N},N}	.
\end{align*}
Proceeding as above, we obtain 
$$
E_{1,z} \lesssim
 \Big(1+\frac{\langle z\rangle^{N_1+1} }{|\Im z|^{N_2}}\Big).
$$
for  integers $N_1,N_2\in \bN$, and this is integrable against ${ |\bar \partial \tilde f (z)| }/
 {|\Im z|} $. For the second expression, we see 
 $$
E_{2,z}\lesssim
\|(\ell_{m_0}-z)^{-1} \psi\|_{S_{loc}^{-m_0},N'_1}
\sum_{j>0}
\|r_{z}^{\# j} \psi_{t_j} \|_{S_{loc}^{-N +m_0},N'_2}.	
$$ 
and by Lemma \ref{lem_est_l-z-1}, 
$$
\|(\ell_{m_0}-z)^{-1} \psi\|_{S_{loc}^{-m_0},N'_1}
\lesssim 
\frac{\langle z\rangle^{N''_1} }{|\Im z|^{N''_1}}
$$
for some integer $N'_1,N'_2,N''_1$ depending on $N$.
Therefore, supposing  $N\geq |m_0|$,
$$
\int_{\bC}
\frac { |\bar \partial \tilde f (z)| }
 {|\Im z|} \ E_{2,z} \ L(dz)
 \lesssim 
\sum_{j>0} \int_{\bC}
 \frac { |\bar \partial \tilde f (z)| }
 {|\Im z|} \ 
 \frac{\langle z\rangle^{N''_1} }{|\Im z|^{N''_1}}
\
 \|r_{z}^{\# j} \psi_{t_j}\|_{S^{0}_{loc},N'_2}
\ L(dz)
$$
and our new assumption on the sequence $(t_j)$ implies that this is finite.
Therefore $I^{(3)}_N$ is finite and 
this concludes the proof.
\end{proof}

\subsection{Proof of Theorem \ref{thm_f(L0)}}
Let $f\in \cM^m(\bR)$ with $m<-1$.
We construct the operator $P_z$ as in Section \ref{subsec_parametrix}, and in particular as in Lemma \ref{lem_Pz} Part (2).
This allows us to consider the expansion of 
its symbol 
$p_z \sim \sum_{j\geq 0} p_{z,-m_0-j}$
 described in Lemma \ref{lem_Pz} and to define the operator 
$$
A:=\frac 1{\pi}
\int_{\bC} \bar \partial \tilde f (z) \
 P_z \ L(dz)	 \ \in \Psi^{- m_0}_{ps}(\Omega),
$$

For every $j,N\in \bN_0$, the form of the symbols $p_{z,-m_0-j}$  allows for the estimate
$$
\|p_{z,-m_0-j}\|_{S^{-m_0-j}_{loc},N}
\lesssim 
\sum_{k=0}^{2j} 
\|d_{j,k}\|_{S^{-j+km_0}_{loc},N} 
\|(\ell_{m_0}-z)^{-1}\psi\|_{S^{-m_0}_{loc},N} ^{1+k}
\lesssim 
\left(\frac{\langle z\rangle }{|\Im z|}\right)^{(N+1)}
+\left(\frac{\langle z\rangle }{|\Im z|}\right)^{(N+1) (2j+1)},
$$
so the integral 
$$
\frac 1{\pi}
\int_{\bC} |\bar \partial \tilde f (z)| \
 \|p_{z,-m_0-j}\|_{S^{-m_0-j}_{loc},N} \ L(dz)
 $$
 is finite.
 Hence, we can define the symbols $a_{mm_0-j}\in S^{-m_0-j}_{loc}(\Omega\times \bR^n)$ for $j=0,1,\ldots$ via
$$
a_{mm_0-j}
:=\frac 1{\pi}
\int_{\bC} \bar \partial \tilde f (z) \
 p_{z,-m_0-j} \ L(dz)	.
$$
Using  Lemma \ref{lem_cauchyF},
 we can compute further 
$$
a_{mm_0-j}
=
\sum_{k=0}^{2j} d_{j,k} \ 
\frac {1}{\pi}
\int_{\bC} \bar \partial \tilde f (z) \
 \frac {L(dz)}{(\ell_{m_0}-z)^{1+k}}
=
\sum_{k=0}^{2j} d_{j,k} \ \frac{(-1)^k}{k!} f^{(k)}(\ell_{m_0}).
$$
This shows that $a_{mm_0-j}$ is in fact in $S^{mm_0-j}_{loc}(\Omega\times \bR^n)$
and also of the form described in Part (1) of Theorem \ref{thm_f(L0)}.
The finiteness of the integral $I^{(2)}_N$ for each $N\in \bN$ in 
Lemma \ref{lem_Pz} Part (2) implies that 
$A-\sum_{j<N}\Op_\Phi (a_{mm_0-j})$ is in $\Psi^{m_0-N}(\Omega)$ for every $N\in \bN$.
Hence $A\in \Psi_{ps}^{mm_0}(\Omega)$, and its symbol $a$  admits the expansions $a\sim \sum_{j\geq 0} a_{mm_0-j}$.

For each $z\in \bC\setminus\bR$, 
let $K_z\in C^\infty(\Omega\times \Omega)$ denote the  integral kernel of $(\cL -z) P_z -\id \in \Psi_{ps}^{-\infty}$. 
For each $y\in \Omega$, $x\mapsto K_z(x,y)$ is smooth and compactly supported in $\Omega$. Lemma \ref{lem_cLOmega'} gives 
$$
\|(\cL-z)^{-1}_x K_z(x,y)\|_{H^{N_1}_{loc}(\Omega_x),N_1}\lesssim 
\frac 1{|\Im z|}
\|K_z(x,y)\|_{H^{N'_1}_{loc}(\Omega_x),N'_1}.
$$
A similar property  holds also for $x\mapsto \partial^\alpha_y (\cL-z)^{-1}_x K_z(x,y) = (\cL-z)^{-1}_x \partial^\alpha_y K_z(x,y)$ for any $\alpha\in \bN_0^n$ and for $y$ in a given compact of $\Omega$. This implies the estimate 
$$
\|(\cL-z)^{-1}_x K_z(x,y)\|_{H^{N_2}_{loc}(\Omega_x\times \Omega_y),N_2}\lesssim 
\frac 1{|\Im z|}
\|K_z(x,y)\|_{H^{N_2'}_{loc}(\Omega_x\times \Omega_y),N_2'},
$$
for every $N_2\in \bN$ and for some $N'_2$ depending on $N_2$.
Since $(\cL-z)^{-1}_x K_z(x,y)$ is the integral kernel of the operator $P_z - (\cL-z)^{-1}$, 
by Lemma \ref{lem_smoothing}, the operator
$P_z - (\cL-z)^{-1}$ is smoothing and properly supported. Furthermore,
it satisfies for every $N\in \bN$
\begin{align*}
\|P_z - (\cL-z)^{-1}\|_{\Psi^{-N},N}
&\lesssim 
\|(\cL-z)^{-1}_x K_z(x,y)\|_{H^{N_1}_{loc}(\Omega_x\times \Omega_y),N_1}
\lesssim 
\frac 1{|\Im z|}
\|K_z(x,y)\|_{H^{N_2}_{loc}(\Omega_x\times \Omega_y),N_2}
\\&\lesssim 
\frac 1{|\Im z|}
\|\id - (\cL-z)Q_z\|_{\Psi^{-N'},N'}	\ ,
\end{align*}
for some $N_1,N_2,N'$.
All the implicit constants above are independent of $z$, 
so by Lemma \ref{lem_Pz}, the integral
$$
\int_\bC |\bar \partial \tilde f(z)|
\left\|(\cL-z)^{-1}-P_z\right\|_{\Psi^{-N},N} L(dz)
\lesssim 
\int_\bC \frac{|\bar \partial \tilde f(z)|}
{|\Im z|}
\left\|\id - (\cL-z)P_z\right\|_{\Psi^{-N'},N'}
L(dz),
$$
is finite.
This implies that the operator
$$
f(\cL)-A = \frac 1\pi \int_\bC \bar \partial \tilde f(z)
\left((\cL-z)^{-1}-P_z\right) L(dz),
$$
is smoothing.
This concludes the proof of Theorem \ref{thm_f(L0)} Part (1).
The rest of the statement follows.

\end{document}